%% file: main.tex
\documentclass[12pt]{article}

\input{packages}
\input{theorem}

\input{title}

\input{paddings}

\input{commands}

\newenvironment{sciabstract}{}

\cftsetindents{section}{0em}{1.7em}

\sectionfont{\large}

\titlespacing*{\section}
{0pt}{5pt}{5pt}

\begin{document}

\maketitle 

\vspace*{-0.7cm}

\vspace*{0.3cm}

\begin{sciabstract}
  \textbf{Abstract.} 
 	We study small eigenvalues of Toeplitz operators on polarized complex projective manifolds.
 	For Toeplitz operators whose symbols are supported on proper subsets, we prove the existence of eigenvalues that decay exponentially with respect to the semiclassical parameter.
 	We moreover, establish a connection between the logarithmic distribution of these eigenvalues and the Mabuchi geodesic between the fixed polarization and the Lebesgue envelope associated with the polarization and the non-zero set of the symbol.
 	As an application of our approach, we also obtain analogous results for Toeplitz matrices.
\end{sciabstract}

\pagestyle{fancy}
\lhead{}
\chead{Toeplitz operators, Lebesgue envelopes and Mabuchi geometry}
\rhead{\thepage}
\cfoot{}

%\fancypagestyle{mypagestyle}{%
%  \fancyhf{}% Clear header/footer
%  \fancyhead[OC]{An Author}% Author on Odd page, Centred
%  \fancyhead[EC]{A titlesdfdsfdsfds}% Title on Even page, Centred
%  \fancyfoot[C]{\thepage}%
%  \renewcommand{\headrulewidth}{.4pt}% Header rule of .4pt
%}
%\pagestyle{mypagestyle}

\newcommand{\Addresses}{{% additional braces for segregating \footnotesize
  \bigskip
  \footnotesize
  \noindent \textsc{Siarhei Finski, CNRS-CMLS, École Polytechnique F-91128 Palaiseau Cedex, France.}\par\nopagebreak
  \noindent  \textit{E-mail }: \texttt{finski.siarhei@gmail.com}.
}} 

\vspace*{0.25cm}

\par\noindent\rule{1.25em}{0.4pt} \textbf{Table of contents} \hrulefill

\vspace*{-1.5cm}

\tableofcontents

\vspace*{-0.2cm}

\noindent \hrulefill

%\vspace*{-0.5cm}

\section{Introduction}\label{sect_intro}
	The primary aim of this paper is to examine the small eigenvalues of Toeplitz operators and their connections with Mabuchi geometry and Lebesgue envelopes.
	\par 
	Throughout the whole article, we fix a complex projective manifold $X$, $\dim X = n$, and an ample line bundle $L$ over it. 
	We fix a smooth positive Hermitian metric $h^L$ on $L$, and denote by ${\textrm{Hilb}}_k(h^L)$ the $L^2$-metric on $H^0(X, L^{\otimes k})$ induced by $h^L$, see (\ref{eq_defn_l2}).
	\par  
	For a fixed $f \in L^{\infty}(X)$ and $k \in \nat^*$, we denote by $T_k(f) \in {\enmr{H^0(X, L^{\otimes k})}}$ the \textit{Toeplitz operator with symbol} $f$.
	Recall that this means that $T_k(f) := B_k \circ M_k(f)$, where $B_k : L^{\infty}(X, L^{\otimes k}) \to H^0(X, L^{\otimes k})$ is the orthogonal (Bergman) projection to $H^0(X, L^{\otimes k})$, and $M_k(f) : H^0(X, L^{\otimes k}) \to L^{\infty}(X, L^{\otimes k})$ is the multiplication map by $f$, acting as $s \mapsto f \cdot s$, $s \in H^0(X, L^{\otimes k})$.
	\par 
	Toeplitz operators have recently found numerous applications in complex and algebraic geometry; see \cite{BerndtProb}, \cite{MaZhangSuperconnBKPubl}, \cite{FinToeplImm}, \cite{FinHNII} and \cite{FinSubmToepl}.
	Their spectral theory of Toeplitz operators is a classical subject, rooted in the seminal work of Boutet de Monvel-Guillemin \cite{BoutGuillSpecToepl}, where the weak convergence of their spectral measures is established. 
	This theory has been recently revisited and extended by the subsequent works on Bergman kernels, including those of Dai-Liu-Ma \cite{DaiLiuMa}, as well as Ma-Marinescu \cite{MaMarBTKah}, \cite{MaHol}.
	\par 
	Significant progress has been made in understanding the smallest eigenvalues of Toeplitz operators.
	The general case reduces to the study of Toeplitz operators with symbols satisfying an additional condition $\essinf_X f = 0$, a condition we will assume throughout.
	We also exclude the trivial case where $f = 0$ almost everywhere from our analysis by simply saying $f \neq 0$.
	\par 
	\begin{sloppypar}
	The asymptotics of the smallest eigenvalue $\lambda_{\min}(T_k(f))$, of $T_k(f)$, as $k \to \infty$, is dictated by the zero set of $f$ and the behavior of $f$ in its vicinity, see \cite{DelepSpec1}, \cite{DelepSpec2}, \cite{AnconaFloch}, \cite{DrewLiuMar1}.
	\end{sloppypar}
	\par 
	\begin{sloppypar}
	Drewitz-Liu-Marinescu in \cite[Theorem 1.23]{DrewLiuMar2} demonstrated that for an arbitrary $f \in L^{\infty}(X)$, $f \neq 0$, $\lambda_{\min}(T_k(f))$ decays no faster than exponentially in $k$, i.e. there is $d > 0$, verifying 
	\begin{equation}\label{eq_min_bndliumar}
		\lambda_{\min}(T_k(f)) \geq \exp(- d k),
	\end{equation}
	for $k \in \nat$ big enough.
	They further conjectured in \cite[Question 5.13]{DrewLiuMar2} that the exponential decay of the first eigenvalue occurs if and only if the symbol of the Toeplitz operator is supported on a proper subset.
	We confirm this conjecture in the present work for \textit{non-pathological} symbols by studying the logarithmic distribution of small eigenvalues.
	\end{sloppypar}
	\par 
	Before describing our results, let us describe the possible pathologies of the symbols.
	We denote $K := \esssupp f$, $Z := f^{-1}(0) \cap K$ and let $NZ := K \setminus Z$; here $NZ$ stands for non-zero.
	Note that the set $Z$ depends on the choice of the measurable representative of $f \in L^{\infty}(X)$; however, the sets $Z$ corresponding to two different representatives differ only by a Lebesgue negligible set, and in all subsequent constructions, we will rely solely on properties of $Z$ that are unaffected by the addition or removal of such Lebesgue negligible sets.
	\par 
	Immediately from the definitions, we see that the closure of $NZ$ coincides with $K$.
	Moreover, $NZ$ is Lebesgue non-negligible.
	Note, however, that the Lebesgue measure of $Z$, which we denote by $\lambda(Z)$, can take an arbitrary value in $[0, \lambda(K)[$.
	For example of a function $f$ with $\lambda(Z)$ non-zero, take the indicator function of a complement of a Smith-Volterra-Cantor set (which is a nowhere dense closed subset of positive measure) or the indicator function of an interior of an Osgood curve in $\mathbb{P}^1$ (a non-self-intersecting curve that has positive area).
	These are exactly the pathologies we will circumvent in \textit{some} of our subsequent results by imposing the condition $\lambda(Z) = 0$.
	Observe that these pathologies might still appear even for smooth functions $f$, as the theorem of Whitney, cf. \cite{CalderonZygmund}, tells that an arbitrary closed subset can be a zero set of a smooth function.
	\par 
	To state our results, we fix a Lebesgue-measurable subset $A \subset X$, which is Lebesgue non-negligible, and consider the \textit{Lebesgue envelope} $h^L_{{\rm{Leb}}, A}$ of a given metric $h^L$ on $L$, defined as
	\begin{equation}\label{defn_leb_env}
		h^L_{{\rm{Leb}}, A}
		=
		\inf \Big\{
			h^L_0 \text{ with psh potential }: h^L_0 \geq h^L \text{ almost everywhere on } A
		\Big\},
	\end{equation}
	where psh above stands for \textit{plurisubharmonic}, and almost everywhere condition above is with respect to the Lebesgue measure on $A$ (induced by the Lebesgue measure on $X$, i.e. of top dimension).
	Observe that adding or removing a Lebesgue negligible set from $A$ does not affect $h^L_{{\rm{Leb}}, A}$.
	\par 
	The envelope construction (\ref{defn_leb_env}), originally introduced by Guedj-Lu-Zeriahi \cite{GuedjLuZeriahEnv} for arbitrary non-pluripolar measures, is studied in Sections \ref{sect_determ}, \ref{sect_leb} in comparison with the psh envelopes of Siciak \cite{SiciakExtremal} and Guedj-Zeriahi \cite{GuedZeriGeomAnal}, see (\ref{defn_env}).
	For the moment, we only mention that similarly to psh envelopes, the assumption that $A$ is Lebesgue non-negligible assures that $h^L_{{\rm{Leb}}, A}$ has a bounded potential.
	But unlike psh envelopes, for which the potential becomes psh only after the upper semi-continous regularization, the potential of $h^L_{{\rm{Leb}}, A}$ is automatically psh, see Proposition \ref{prop_leb_env_id}.
	\par 
	Following the analogous notion in complex pluripotential theory, cf. \cite{KlimekBook}, we define a pair $(A, h^L)$ to be \textit{Lebesgue pluriregular} if $h^L_{{\rm{Leb}}, A}$ is continuous, see Proposition \ref{prop_example_leb_plur} for examples.
	\begin{sloppypar}
	\begin{thm}\label{thm_regul}
		Let $f \in L^{\infty}(X)$, $f \neq 0$, be such that $(K, h^L)$ is Lebesgue pluriregular and $\lambda(Z) = 0$.
		We let $c(f) := \max_{x \in X} \log(h^L(x) / h^L_{{\rm{Leb}}, K}(x))$.
		Then for any $\epsilon > 0$, there is $k_0 \in \nat$, such that for any $k \geq k_0$, we have
		\begin{equation}
			\exp(- (c(f) + \epsilon) k) \leq \lambda_{\min}(T_k(f)) \leq \exp(- (c(f) - \epsilon) k).
		\end{equation}
	\end{thm}
	\begin{rem}\label{rem_thm_regul}
		Under the given assumptions, the Lebesgue envelope coincides with the psh envelope, see Proposition \ref{prop_plurireg}. 
		However, the distinction between these two notions becomes significant in the next result, where no assumption on Lebesgue pluriregularity is made.
	\end{rem}
	\end{sloppypar}
	\par 
	Another essential ingredient for this paper is the concept of a Mabuchi geodesic, which provides a specific method for constructing paths between metrics on $L$ with bounded plurisubharmonic potentials. 
	For the necessary background, see Section \ref{sect_mabuch}.
	We denote by $h^L_t$, $t \in [0, 1]$, the Mabuchi geodesic between $h^L_0 := h^L$ and $h^L_1 := h^L_{{\rm{Leb}}, NZ}$.
	While in general, $h^L_t$ is not smooth in $t \in [0, 1]$, one can still define its time derivative at $t := 0$ due to convexity.
	We denote by
	\begin{equation}\label{defn_phi}
		\phi(h^L, NZ) := (h^L_t)^{-1} \frac{d}{dt} h^L_t|_{t = 0}
	\end{equation}
	the speed of the geodesic at $t = 0$.
	We shall prove in Section \ref{sect_mabuch} that $\phi(h^L, NZ)$ is a non-positive bounded function, it is strictly negative away from $K$, and it vanishes on the Lebesgue dense points of $NZ$ (see (\ref{eq_leb_dens}) for the definition of Lebesgue dense points).
	In particular, the support of $\phi(h^L, NZ)$ and $f$ intersect at most along the topological boundary of $K$.
	The main result of this paper, detailed below, asserts that the logarithmic distribution of the small eigenvalues is determined by this function $\phi(h^L, NZ)$.
	\par 
	\begin{thm}\label{thm_distr}
		For any $f \in L^{\infty}(X)$, $f \neq 0$, and continuous $g: \real \to \real$, we have
		\begin{equation}
			\lim_{k \to \infty} \frac{1}{\dim H^0(X, L^{\otimes k})} \sum_{\lambda \in {\rm{Spec}} (T_k(f))} g\Big( \frac{\log(\lambda)}{k} \Big)
			=
			\frac{\int_X g(\phi(h^L, NZ)) c_1(L, h^L)^n}{\int_X c_1(L, h^L)^n}.
		\end{equation}
	\end{thm}
	\par 
	Theorems \ref{thm_regul} and \ref{thm_distr} collectively imply that the asymptotic shape of eigenvalues having exponential decay depends solely on the non-zero set $NZ$ of the symbol and not on the nature of its decay in the vicinity of the boundary of $NZ$.
	This goes in sharp contrast with the properties of the smallest eigenvalue of a Toeplitz operator with a continuous symbol when $f^{-1}(0)$ consists of a single point, see \cite[\S 5.5, 5.6]{DrewLiuMar2}.
	\par 
	We now briefly discuss some qualitative consequences of the above results.
	The following result gives an answer to \cite[Question 5.13]{DrewLiuMar2}.
	\begin{cor}\label{cor_exist}
		For any $f \in L^{\infty}(X)$, $f \neq 0$, verifying $\lambda(Z) = 0$, we have $K \neq X$ if and only if there is $c > 0$, such that $\lambda_{\min}(T_k(f)) \leq \exp(- ck)$ for $k \in \nat$ big enough.
		Moreover, the proportion of exponentially small eigenvalues equals to the relative volume of $X \setminus K$, i.e.
		\begin{equation}\label{eq_exist}
			\lim_{\epsilon \to +0} \lim_{k \to \infty} \frac{\# \{ \lambda \in {\rm{Spec}} (T_k(f)) : \lambda \leq \exp(- \epsilon k) \} }{ \# {\rm{Spec}} (T_k(f)) }
			=
			\frac{\int_{X \setminus K} c_1(L, h^L)^n}{\int_X c_1(L, h^L)^n}.
		\end{equation}
	\end{cor}
	\begin{rem}
		The statement (\ref{eq_exist}) with Proposition \ref{prop_nonzero_eig} show that the fraction of eigenvalues that are exponentially small matches the fraction of eigenvalues that tend to zero.
	\end{rem}
	The next result states that in a certain sense the distribution of small eigenvalues depends solely on the geometry of the manifold away from the essential support of the symbol.
	For $A \subset X$, we denote below by $i(A)$ the subset of interior points $A$.
	\begin{cor}\label{cor_local}
		Assume that for two complex manifolds $X$, $X'$; $f \in L^{\infty}(X)$, $f' \in L^{\infty}(X')$ and ample line bundles $L$, $L'$ with positive metrics $h^L$, $h^{L'}$ the following holds.
		For $K := \esssupp f$, $K' := \esssupp f'$, there are neighborhoods $U$, $U'$ of $X \setminus i(K)$ and $X' \setminus i(K')$ respectively, so that there is a biholomorphism $p : U \to U'$, which extends to the holomorhic map between $L$ and $L'$, preserving $h^L$ and $h^{L'}$.
		Then the logarithmic distribution (in the sense of Theorem \ref{thm_distr}) of the small eigenvalues of respective Toeplitz operators $T_k(f)$ and $T_k(f')$ are asymptotically identical.
		If, moreover, both $(K, h^L)$ and $(K', h^{L'})$ are Lebesgue pluriregular, then the exponents of the growth of the smallest eigenvalues of  $T_k(f)$ and $T_k(f')$ from Theorem \ref{thm_regul} are identical.
	\end{cor}
	Our last statement says that while by Theorem \ref{thm_regul} the smallest eigenvalue decays exponentially if $K \neq X$, upon changing the polarization, the exponent can be made arbitrary small.
	\begin{cor}\label{cor_small_exp}
		Let $f \in L^{\infty}(X)$, $f \neq 0$, be such that $(K, h^L)$ is Lebesgue pluriregular and $\lambda(Z) = 0$.
		Then for any $\epsilon > 0$, there is a positive smooth metric $h^L_0$ on $L$, such that over $K$, we have $h^L \geq h^L_0 \geq \exp(- \epsilon) h^L$, and the smallest eigenvalue $\lambda_{\min}(T^0_k(f))$, of the Toeplitz operator $T^0_k(f)$, associated with $h^L_0$ and $f$, satisfies $\exp(- \epsilon k) \leq \lambda_{\min}(T^0_k(f))$ for $k$ big enough.
	\end{cor}
	\par 
	To conclude, we briefly outline our approach to the theorems mentioned above. 
	Our approach relies on analyzing the operators $\frac{1}{k} \log T_k(f)$; we show specifically that these operators coincide asymptotically with the Toeplitz operator associated with the symbol $\phi(h^L, NZ)$. 
	Once this is established, the above results would follow directly from the detailed analysis of $\phi(h^L, NZ)$, and some results on the spectrum of Toeplitz operators.
	\par 
	To relate $\log T_k(f)$ with Toeplitz operators, we interpret $\log T_k(f)$ as the transfer operator between the $L^2$-norm ${\textrm{Hilb}}_k(h^L)$ associated with $h^L$ and the $L^2$-norm associated with $h^L$ and the weight $f$. 
	We then extend a result of Berman-Boucksom-Witt Nyström \cite{BerBoucNys} concerning the asymptotic study of the $L^2$-norms associated with weights to the case when the measure induced by the weight is not necessarily determining and the essential support is not necessarily pluriregular.
	The relation between $\log T_k(f)$ and Toeplitz operators is then derived by extending the previous result of the author \cite{FinSubmToepl} on the asymptotics of the transfer operators between $L^2$-norms associated with smooth volume forms.
	\par 
	This article is organized as follows.
	In Section \ref{sect_log_toepl}, we establish the main results of the paper, modulo the statement about the asymptotics of $\log T_k(f)$.
	In Section \ref{sect_determ}, we study the non-negligible psh envelopes associated with an arbitrary non-pluripolar probability measure and compare them with psh envelopes.
	In Section \ref{sect_leb}, we specialize this theory to the case of Lebesgue envelopes. 
	In Section \ref{sect_bm}, we study the Bernstein-Markov property for the Lebesgue measure.
	In Sections \ref{sect_transf} and \ref{sect_bern_mark}, we establish that the operator $\frac{1}{k} \log T_k(f)$ is asymptotically Toeplitz.
	In Section \ref{sect_mabuch}, we study the Mabuchi geodesic connecting a metric with its Lebesgue envelope.
	In Section \ref{sect_growth_balls}, we apply our methods to study the growth of balls of holomorphic sections associated with $L^2$-norms supported on measurable subsets, extending a previous result of Berman-Boucksom \cite{BermanBouckBalls}.
	In Section \ref{sect_gen_toepl}, we prove the analogues of Theorems \ref{thm_regul}, \ref{thm_distr} for generalized Toeplitz operators and Toeplitz matrices.
	
	\paragraph{Notation.}  
	For $f : X \to \real$, we denote by $f^*$ (resp. $f_*$) the upper (resp. lower) semi-continuous regularization of $f$.
	Similar notations will be used for metrics on line bundles.
	For $f, g \in L^{\infty}(X)$, by $f = g$ we mean that the measurable representatives of $f, g$ coincide almost everywhere.
	Similarly, we say that $f$ is continuous if we can find a continuous representative in the class of $f$.
	\par 
	By a positive Hermitian metric on a line bundle we mean a smooth Hermitian metric with strictly positive curvature.
	Upon fixing a positive metric $h^L_0$ on $L$, one can identify the space of positive metrics on $L$ with the space of Kähler potentials of $\omega := 2 \pi c_1(L, h^L_0)$, consisting of $u \in \ccal^{\infty}(X, \real)$, such that $\omega_u := \omega + \imun \partial \dbar u$ is strictly positive, through the map 	
	\begin{equation}\label{eq_pot_metr_corr}
		u \mapsto h^L := e^{- u} \cdot h^L_0.
	\end{equation}
	The function $u$ will be called the potential of the metric $h^L$ (with respect to $h^L_0$).
	When working locally, we will, by a slight abuse of notation, also refer to a function $u$ as a potential if, in a \textit{holomorphic} local frame, $h^L$ can be expressed as $e^{- u}$.
	\par 
	For a Kähler form $\omega$ on $X$, we denote by ${\rm{PSH}}(X, \omega)$ the space of $\omega$-psh potentials, consisting of functions $\psi : X \to [- \infty, +\infty[$, which are locally the sum of a psh function and of a smooth function so that the $(1, 1)$-current $\omega + \imun \partial \dbar \psi$ is positive.
	We say that $\psi$ is strictly $\omega$-psh if the $(1, 1)$-current $\omega + \imun \partial \dbar \psi$ is strictly positive, i.e. there is $\epsilon > 0$, such that $\omega + \imun \partial \dbar \psi \geq \epsilon \omega$. 
	\par 
	\begin{sloppypar}
	For a metric $h^L$ on $L$ with bounded psh potential and a positive Borel measure $\mu$ on $X$, we denote by ${\textrm{Hilb}}_k(h^L, \mu)$ the positive semi-definite form on $H^0(X, L^{\otimes k})$ defined for arbitrary $s_1, s_2 \in H^0(X, L^{\otimes k})$ as follows
	\begin{equation}\label{eq_defn_l2}
			\langle s_1, s_2 \rangle_{{\textrm{Hilb}}_k(h^L, \mu)} = \int_X \langle s_1(x), s_2(x) \rangle_{{h^{L^{\otimes k}}}} \cdot d \mu(x).
	\end{equation}
	Remark that when $\mu$ is a non-pluripolar Borel measure (i.e. a Borel measure not charging the pluripolar subsets), the above form is positive definite.
	When the measure $\mu$ is induced by a non-negative $L^{\infty}(X)$-section $\chi$ of $\wedge^{n, n} T^* X$, we denote the associated form by ${\textrm{Hilb}}_k(h^L, \chi)$.
	For brevity, we use the notations ${\textrm{Hilb}}_k(h^L)$ for ${\textrm{Hilb}}_k(h^L, c_1(L, h^L)^n)$ and ${\textrm{Hilb}}_k(A, h^L)$ for ${\textrm{Hilb}}_k(h^L, 1_A \cdot c_1(L, h^L)^n)$, where $A \subset X$ is a fixed measurable subset and $1_A$ is the indicator function of it.
	For a closed Lebesgue non-negligible subset $K \subset X$, we denote by ${\textrm{Ban}}_k^{\infty}(K, h^L)$ the $L^{\infty}(K)$-norm on $H^0(X, L^{\otimes k})$ induced by $h^L$. 
	For brevity, we denote ${\textrm{Ban}}_k^{\infty}(X, h^L)$ by ${\textrm{Ban}}_k^{\infty}(h^L)$.
	\end{sloppypar}
	\par 
	The $p$-Schatten norm,  $\| \cdot \|_p$, $p \in [1, +\infty[$, is defined for an operator $A \in {\enmr{V}}$, of a finitely-dimensional Hermitian vector space $(V, H)$ as $\| A \|_p = (\frac{1}{\dim V} {\rm{Tr}}[|A|^p])^{\frac{1}{p}}$, where $|A| := (A A^*)^{\frac{1}{2}}$.
	It is evident that for any $p \in [1, +\infty[$, we have $\| \cdot \|_p \leq \| \cdot \|$, where $\| \cdot \|$ is the operator norm.
	For trivial reasons, we sometimes denote $\| \cdot \|$ by $\| \cdot \|_{+ \infty}$.
	\par 
	When we speak of a Lebesgue measure on a measurable subset $A \subset X$, we mean the ambient Lebesgue measure (i.e. of top dimension to the effect that the Lebesgue measure of a smooth curve in $\mathbb{P}^1$ is zero).
	For a measure $\mu$ on $X$ and a measurable function $f : X \to \real$, we denote by $\esssup_{\mu} f$ the essential supremum of $f$ with respect to $\mu$.
	For brevity, for a measurable subset $A \subset X$ we also note $\esssup_A f$ the essential supremum of $f$ with respect to the Lebesgue measure on $A$.
	Similar notations are used for the essential infimum.
	
	\paragraph{Acknowledgement.} 
	The author acknowledges the support of CNRS, École Polytechnique and the partial support of ANR projects QCM (ANR-23-CE40-0021-01), AdAnAr (ANR-24-CE40-6184) and STENTOR (ANR-24-CE40-5905-01).
	This work was completed during a visit to the Max Planck Institute for Mathematics in Bonn, and the author expresses gratitude to the institute's members for their warm hospitality and the institute for the support. 
	Special thanks are also due to Bingxiao Liu from the University of Cologne for several insightful discussions and for bringing \cite[Question 5.13]{DrewLiuMar2} to the author's attention.
	
	\section{Logarithm of a Toeplitz operator is asymptotically Toeplitz}\label{sect_log_toepl}
	This section is devoted to establishing all the results announced in the Introduction.
	To achieve this, we use a result concerning the asymptotics of the logarithm of a Toeplitz operator, the proof of which we defer to Section \ref{sect_transf}.
	For this, we need \cite[Definition 1.5]{FinSubmToepl}, which we recall below.
	\begin{defn}\label{defn_toepl_schatten}
		We say that $T_k \in {\enmr{H^0(X, L^{\otimes k})}}$, $k \in \nat$, form an \textit{asymptotically Toeplitz operator of Schatten class with symbol $f \in L^{\infty}(X)$} if there is $C > 0$, so that for any $k \in \nat$, we have $\| T_k \| \leq C$, and for any $\epsilon > 0$, $p \in [1, +\infty[$, there is $k_0 \in \nat$, such that for any $k \geq k_0$,
		\begin{equation}\label{eq_thm_log2}
			\big\|
				T_k
				-
				T_k(f)
			\big\|_p
			\leq
			\epsilon.
		\end{equation}
		If we can even take $p = + \infty$ above and $f$ is continuous, we say that $T_k$, $k \in \nat$, form an \textit{asymptotically Toeplitz operator with symbol $f$}, cf. \cite{MaHol}.
	\end{defn}
	\par 
	We fix $f \in L^{\infty}(X)$, $f \neq 0$, so that $\essinf_X f = 0$, and use the notations $K$, $Z$ and $NZ$ as in Introduction.
	Recall that the bounded function $\phi(h^L, NZ)$ was introduced in (\ref{defn_phi}). 
	\par 
	The following result is one of the main contributions of the current article. 
	\begin{thm}\label{thm_log2}
		The sequence $\frac{1}{k} \log (T_k(f))$, $k \in \nat^*$, forms an asymptotically Toeplitz operator of Schatten class with symbol $\phi(h^L, NZ)$.
		If, moreover, $(NZ, h^L)$ is Lebesgue pluriregular, then $\frac{1}{k} \log (T_k(f))$, $k \in \nat^*$, even forms an asymptotically Toeplitz operator with symbol $\phi(h^L, NZ)$.
	\end{thm}
	\begin{rem}
		If we had $\essinf_X f > 0$, as shown in \cite[Proposition 5.10]{FinSubmToepl}, a markedly different phenomenon would arise: the sequence $\log (T_k(f))$, $k \in \nat^*$, (note the absence of the factor $\frac{1}{k}$ in front of $\log (T_k(f))$) would form an asymptotically Toeplitz operator of Schatten class with symbol $\log(f)$, see also \cite{BoutGuillSpecToepl}, \cite{BordMeinSchli} and \cite{MaMarBTKah} for related results.
	\end{rem}
	\par 
	Theorem \ref{thm_log2} will be proved in Section \ref{sect_transf}. 
	Let us demonstrate how it leads to Theorems \ref{thm_regul} and \ref{thm_distr}. 
	For this, we will need the following result concerning Mabuchi geodesics. 
	\par 
	We fix a positive smooth metric $h^L_0$ on $L$, and a metric $h^L_1$ on $L$ with a bounded psh potential.
	Let $h^L_t$, $t \in [0, 1]$, be the Mabuchi geodesic between $h^L_0$ and $h^L_1$.
	We denote 
	\begin{equation}\label{eq_speed_mab_geod}
		\phi(h^L_0, h^L_1) := - (h^L_t)^{-1} \frac{d}{dt} h^L_t|_{t = 0}.
	\end{equation}
	The following result is established in \cite[Proposition 3.1]{FinSubmToepl}, \cite[Lemma 4.7]{FinSecRing} and Proposition \ref{prop_mab_contact_set}.
	\begin{sloppypar}
	\begin{prop}\label{prop_cont_speed}
		If $h^L_1$ is continuous, then $\phi(h^L_0, h^L_1)$ is continuous, and $\max_X |\phi(h^L_0, h^L_1)| = \max_X |\log (h^L_0 / h^L_1)|$.
		If, moreover, $h^L_0 \geq h^L_1$, then $\phi(h^L_0, h^L_1) \geq 0$.
	\end{prop}
	\end{sloppypar}
	
	\begin{proof}[Proof of Theorem \ref{thm_regul}]
		By the Lebesgue pluriregularity of $(K, h^L)$ and Proposition \ref{prop_cont_speed}, $\phi(h^L, K)$ is continuous.
		Remark also that since $\lambda(Z) = 0$ by our assumption, we have $\phi(h^L, K) = \phi(h^L, NZ)$.
		Clearly, Theorem \ref{thm_regul} simply says that the biggest eigenvalue of $- \frac{1}{k} \log (T_k(f))$ converges, as $k \to \infty$, to $c(f)$ in the notations of Theorem \ref{thm_regul}.
		It is a standard fact, cf. \cite[Theorem 5.1]{BarrMa}, that the biggest eigenvalue of a Toeplitz operator with a continuous symbol converges to the maximal value of the symbol.
		From this and the second part of Theorem \ref{thm_log2}, we deduce that the biggest eigenvalue of $- \frac{1}{k} \log (T_k(f))$ converges to $\max_X \phi(h^L, NZ)$.
		By this and the second statement of Proposition \ref{prop_cont_speed}, we deduce Theorem \ref{thm_regul}.
	\end{proof}
	\par

	Let us now recall \cite[Proposition 5.12]{FinSubmToepl}, which generalizes the weak convergence from \cite{BoutGuillSpecToepl}.
	\begin{prop}\label{prop_weak_convtt}
		The spectral measures of a sequence of operators $T_k \in {\enmr{H^0(X, L^{\otimes k})}}$, $k \in \nat$, forming an asymptotically Toeplitz operator of Schatten class with symbol $f \in L^{\infty}(X)$, converge weakly, as $k \to \infty$, to $f_* (\frac{1}{\int_X c_1(L)^n} c_1(L, h^L)^n )$.
	\end{prop}

	\begin{proof}[Proof of Theorem \ref{thm_distr}]
		It is a direct consequence of Theorem \ref{thm_log2} and Proposition \ref{prop_weak_convtt}.
	\end{proof}
	
	To establish Corollary \ref{cor_exist}, we need to establish some further results.
	
	\par 
	
	\begin{prop}\label{prop_nonzero_eig}
		Assume that a sequence of self-adjoint operators $T_k \in {\enmr{H^0(X, L^{\otimes k})}}$, $k \in \nat$, forms a Toeplitz operator of Schatten class with symbol $f$.
		Then the proportion of asymptotically non-null eigenvalues of $T_k$ equals to the relative volume of $X \setminus NZ$, i.e.
		\begin{equation}\label{eq_exist_2}
			\lim_{\epsilon \to +0} \lim_{k \to \infty} \frac{\# \{ \lambda \in {\rm{Spec}} (T_k) : \lambda \leq \epsilon \} }{ \# {\rm{Spec}} (T_k) }
			=
			\frac{\int_{X \setminus NZ} c_1(L, h^L)^n}{\int_X c_1(L, h^L)^n}.
		\end{equation}
	\end{prop}
	\begin{sloppypar}
	\begin{proof}
		Consider a sequence $\epsilon_l > 0$, $l \in \nat$, which tends to $0$, as $l \to \infty$, and such that all $\epsilon_l$ are in the continuity set of the cumulative distribution function associated with the measure $f_* (\frac{1}{\int_X c_1(L)^n} c_1(L, h^L)^n )$ on $\real$ (note that since the cumulative distribution function is monotone, its set of discontinuity points is at most countable, which allows a choice of the sequence $\epsilon_l$).
		By the Portmanteau theorem and the choice of $\epsilon_l$, the weak convergence of spectral measures from Proposition \ref{prop_weak_convtt} implies that for any $l \in \nat$,  for the subsets $U_l := \{ x \in X : f(x) > \epsilon_l \}$, we have
		\begin{equation}\label{eq_exist_1}
			\lim_{k \to \infty} \frac{\# \{ \lambda \in {\rm{Spec}} (T_k) : \lambda \leq \epsilon_l \} }{ \# {\rm{Spec}} (T_k) }
			=
			\frac{\int_{X \setminus U_l} c_1(L, h^L)^n}{\int_X c_1(L, h^L)^n}.
		\end{equation}
		By taking a limit $l \to \infty$ in (\ref{eq_exist_1}), using the fact that $\cup_{l \in \nat} U_l = \{ x \in X : f(x) > 0 \}$ and applying Proposition \ref{prop_leb_negl}, we deduce (\ref{eq_exist}), which finishes the proof.
	\end{proof}
	\end{sloppypar}
	
	The following results from complex pluripotential theory are established in Sections \ref{sect_leb} and \ref{sect_mabuch}.
	\begin{prop}\label{prop_leb_env_all}
		For any continuous metric $h^L$ with a psh potential, we have $h^L_{{\rm{Leb}}, X} = h^L$.
		In particular, $(X, h^L)$ is Lebesgue pluriregular. 
		Moreover, for any measurable subset $A \subset X$, such that $(A, h^L)$ is Lebesgue pluriregular, we have $h^L_{{\rm{Leb}}, A} = h^L$ over $\overline{A}$.
	\end{prop}
	\begin{rem}
		See also \cite[Proposition 2.11]{GuedjLuZeriahEnv} for a related result.
	\end{rem}
	\begin{prop}\label{prop_leb_negl}
		The symmetric difference between $X \setminus K$ and $\{ x \in X : \phi(h^L, K)(x) < 0 \}$ is Lebesgue-negligible. 
	\end{prop}
	
	\begin{sloppypar}
	\begin{proof}[Proof of Corollary \ref{cor_exist}]
		Immediately from Theorem \ref{thm_regul} and Proposition \ref{prop_leb_env_all}, we see that if $\esssupp f = X$, then the smallest eigenvalue of $T_k(f)$ decays subexponentially.
		\par 
		To establish Corollary \ref{cor_exist} in full, now it suffices to prove (\ref{eq_exist}).
		Remark that (\ref{eq_exist}) can be reformulated as follows
		\begin{equation}\label{eq_exist2}
			\lim_{\epsilon \to +0} \lim_{k \to \infty} \frac{\# \{ \lambda \in {\rm{Spec}} (- \frac{1}{k} \log (T_k(f))) : \lambda \leq \epsilon \} }{ \# {\rm{Spec}} (- \frac{1}{k} \log (T_k(f))) }
			=
			\frac{\int_{X \setminus K} c_1(L, h^L)^n}{\int_X c_1(L, h^L)^n}.
		\end{equation}
		But it follows directly from Theorem \ref{thm_log2} and Propositions \ref{prop_nonzero_eig}, \ref{prop_leb_negl}, as by our standing assumption $\lambda(Z) = 0$, we obviously have $\phi(h^L, NZ) = \phi(h^L, K)$.
	\end{proof}
	\end{sloppypar}
	In Section \ref{sect_mabuch}, we shall also establish the following result. 
	\begin{prop}\label{prop_local_geod}
		In the notations and assumptions from Corollary \ref{cor_local}, we have $\phi(h^L, NZ) = p^* \phi(h^{L'}, NZ')$, where $NZ'$ is the non-zero set of $f'$.
	\end{prop}
	
	\begin{proof}[Proof of Corollary \ref{cor_local}]
		It follows immediately from Theorems \ref{thm_regul}, \ref{thm_distr} and Proposition \ref{prop_local_geod}.
	\end{proof}
	\begin{proof}[Proof of Corollary \ref{cor_small_exp}]
		By the Demailly regularization theorem and the fact that $h^L_{{\rm{Leb}}, K}$ has a psh potential, cf. Proposition \ref{prop_leb_env_id}, there is a sequence of positive metrics $h^L_i$, $i \in \nat^*$, on $L$, increasing pointwise towards $h^L_{{\rm{Leb}}, K}$, see \cite{DemRegul}, \cite{GuedZeriGeomAnal}.
		Since $h^L_{{\rm{Leb}}, K}$ is continuous by the Lebesgue pluriregularity assumption, the convergence is uniform by Dini's theorem.
		Hence, we can find $i \in \nat$, so that for $h^L_0 := h^L_i$, the following bound is satisfied
		\begin{equation}\label{eq_bnd_env_reg}
			h^L_{{\rm{Leb}}, K} \geq h^L_0 \geq \exp(-\epsilon) h^L_{{\rm{Leb}}, K}.
		\end{equation}
		We claim that $h^L_0$ will satisfy the assumptions of Corollary \ref{cor_local}.
		\par 
		To see this, we first note that since $h^L$ was initially fixed to be positive and $(K, h^L)$ is Lebesgue pluriregular, by Proposition \ref{prop_leb_env_all}, over $K$ we have $h^L_{{\rm{Leb}}, K} = h^L$, which implies the first condition on $h^L_0$ from Corollary \ref{cor_small_exp} by (\ref{eq_bnd_env_reg}).
		Second, since the formation of psh envelopes clearly preserves the order, and the Lebesgue envelope associated with $h^L_{{\rm{Leb}}, K}$ and $K$ obviously equals $h^L_{{\rm{Leb}}, K}$, for the Lebesgue envelope $h^L_{0, {\rm{Leb}}, K}$ associated with $h^L_0$ and $K$, by (\ref{eq_bnd_env_reg}), we have $h^L_{{\rm{Leb}}, K} \geq h^L_{0, {\rm{Leb}}, K} \geq \exp(-\epsilon) h^L_{{\rm{Leb}}, K}$.
		We deduce that $c_0(f) := \max_{x \in X} \log(h^L_0(x) / h^L_{0, {\rm{Leb}}, K}(x))$ satisfies $c_0(f) \leq \epsilon$.
		An application of Theorem \ref{thm_regul} finishes the proof.
	\end{proof}
	
	\section{Determining measures and non-negligible psh envelopes}\label{sect_determ}
	This section studies the non-negligible psh envelopes associated with a non-pluripolar Borel measure and compares them with classical psh envelopes. 
	As an application, we interpret the determining property of a measure through the identity between these envelopes.
	\par 
	Although our primary motivation concerns the Lebesgue measures on measurable subsets of $X$, which forms a very particular class of non-pluripolar Borel measures, we have chosen to develop the theory in a more general setting.
	The proofs of the results do not simplify dramatically in the Lebesgue case, but by presenting the general picture first, we hope to make the distinctive features of the Lebesgue case more apparent.
	This more general setting will also play a crucial role in Section \ref{sect_gen_toepl}, where we study generalized Toeplitz operators and Toeplitz matrices.
	\par 
	We denote by $\mu$ a non-pluripolar Borel probability measure on $X$ with support $K$.
	For a continuous metric $h^L$ on $L$, following \cite{GuedjLuZeriahEnv}, we define the \textit{non-negligible psh envelope} $h^L_{\mu}$ as
	\begin{equation}\label{defn_nonnegl_env}
		h^L_{\mu}
		=
		\inf \Big\{
			h^L_0 \text{ with psh potential }: h^L_0 \geq h^L \text{$\mu$-almost everywhere on } K
		\Big\}.
	\end{equation}
	Remark that this definition depends only on the absolutely continuous class of the measure $\mu$.
	\par 
	We fix a non-pluripolar subset $E \subset X$. 
	Following Siciak \cite{SiciakExtremal}, Guedj-Zeriahi \cite{GuedZeriGeomAnal}, we define the \textit{psh envelope} $h^L_{E}$ associated with $E$ as
	\begin{equation}\label{defn_env}
		h^L_E
		=
		\inf \Big\{
			h^L_0 \text{ with psh potential }: h^L_0 \geq h^L \text{ over } E
		\Big\}.
	\end{equation}
	The non-pluripolarity of $E$ assures that the metric $h^L_E$ has a bounded potential, cf. \cite[Theorem 9.17]{GuedZeriGeomAnal}.
	A pair $(E, h^L)$ is called \textit{pluriregular} if $h^L_E$ is continuous, cf. \cite[p. 186]{KlimekBook}. 
	\par 
	Note that while the inequality $h^L_E \geq h^L$ on $E$ is immediate, it is not obvious a priori that $h^L_{\mu}$ admits a potential that is bounded somewhere.
	The following result due to Guedj-Lu-Zeriahi \cite[Proposition 2.10]{GuedjLuZeriahEnv}, however, establishes this.
	\begin{prop}\label{prop_env_gen_fact}
		The metric $h^L_{\mu}$ has a bounded psh potential and there is $E \subset K$, $\mu(E) = 0$, such that $h^L_{\mu} = h^L_{K \setminus E}$. 
	\end{prop}
	\begin{rem}\label{rem_leb_env_id}
		a) The set $E$ is not uniquely determined.
		 The reader can verify that if $E$ satisfies the assumptions of Proposition \ref{prop_env_gen_fact}, then so does any enlarged set of the form $E \cup F$ where $F \subset K$ is an arbitrary subset with $\mu(F) = 0$.
		 Note, however, that directly from Zorn's lemma and \cite[Proposition 9.19.3]{GuedjZeriahBook}, there is \textit{a minimal} $E \subset K$, $\mu(E) = 0$, such that $h^L_{\mu} = h^L_{K \setminus E *}$.  
		 \par 
		 b) As over $K \setminus E$, we trivially have $h^L_{K \setminus E} \geq h^L$, we infer by Proposition \ref{prop_env_gen_fact}  that $h^L_{\mu} \geq h^L$ over $K \setminus E$.
		 In particular, the infimum from (\ref{defn_nonnegl_env}) has the minimal element.
		If $h^L$ has moreover a psh potential, by the trivial bound $h^L \geq h^L_{\mu}$, then we even have $h^L_{\mu} = h^L$ over $K \setminus E$.
	\end{rem}
	We fix a positive metric $h^L_0$ on $L$, denote $\omega := 2 \pi c_1(L, h^L_0)$, and let $\phi$ (resp. $\phi_E$, $\phi_{\mu}$) be the potential of $h^L$ (resp. $h^L_E$, $h^L_{\mu}$), cf. (\ref{eq_pot_metr_corr}).
	Directly from the definitions, we can rewrite
	\begin{equation}\label{eq_phi_mu_form1}
	\begin{aligned}
		&
		\phi_{\mu}
		=
		\sup \Big\{
			\phi_0 \in {\rm{PSH}}(X, \omega) : \phi_0 \leq \phi \text{ $\mu$-almost everywhere on } K
		\Big\},
		\\
		&
		\phi_E
		=
		\sup \Big\{
			\phi_0 \in {\rm{PSH}}(X, \omega) : \phi_0 \leq \phi \text{ on } E
		\Big\}.
	\end{aligned}
	\end{equation}
	Immediately from this, we then have
	\begin{equation}\label{eq_phi_mu_form}
		\phi_{\mu}
		=
		\sup \Big\{
			\phi_{K \setminus F} : \mu(F) = 0
		\Big\}.
	\end{equation}
	The reader will then check that the second part of Proposition \ref{prop_env_gen_fact} can be reformulated as follows:
	for any $\psi \in {\rm{PSH}}(X, \omega)$, we have
	\begin{equation}\label{eq_esssup_phi_mu_ident}
		\esssup_{\mu} (\psi - \phi) 
		=
		\sup_{X}  (\psi - \phi_{\mu}).
	\end{equation}
	Below we provide an independent proof of Proposition \ref{prop_env_gen_fact} for the convenience of the reader.
	\begin{proof}[Proof of Proposition \ref{prop_env_gen_fact}]
		First, we establish that $\phi_{\mu}$ is uniformly bounded. 
		Suppose, for contradiction, that there exists a sequence of sets $E_i$, $i \in \nat$, with $\mu(E_i) = 0$, such that $\sup_X \phi_{K \setminus E_i} \to +\infty$.
		Remark that for $E := \cup_{i \in \nat} E_i$, we have $\phi_{K \setminus E_i} \leq \phi_{K \setminus E}$ for any $i \in \nat$, and as a consequence, $\sup_X \phi_{K \setminus E} = +\infty$.
		However, since $\mu(E) = 0$, we have $\mu(K \setminus E) = 1$, and so $K \setminus E$ is non-pluripolar by our assumption on $\mu$.
		But then from \cite[Theorem 9.17]{GuedZeriGeomAnal}, $\sup_X \phi_{K \setminus E} < +\infty$, which gives a contradiction. 
		We conclude that $\phi_{\mu}$ is uniformly bounded.
		\par 
		We will now show that $\phi_{\mu} \in {\rm{PSH}}(X, \omega)$.
		For this, we first remark that for any non-pluripolar $F \subset X$, we have $\phi_{F}^* = \phi_{F}$ away from a pluripolar (or empty) subset by \cite[Theorem 4.42]{GuedjZeriahBook}.
		We denote by $G \subset X$ this subset, and then we see that $\phi_{F}^*$ is one of the contenders for the supremum in the definition of $\phi_{F \setminus G}$, which immediately implies that $\phi_{F}^* \leq  \phi_{F \setminus G}$. 
		Since $\mu$ is non-pluripolar, we conclude by (\ref{eq_phi_mu_form}) and the above
		\begin{equation}\label{eq_phi_mu}
			\phi_{\mu}
			=
			\sup \Big\{
				\phi_{K \setminus E}^* : \mu(E) = 0
			\Big\}.
		\end{equation}
		\par 
		By the uniform boundness of $\phi_{\mu}$, (\ref{eq_phi_mu}), the fact that $\phi_{K \setminus E}^* \in {\rm{PSH}}(X, \omega)$, cf. \cite[Theorem 9.17]{GuedjZeriahBook}, and \cite[Proposition I.4.24]{DemCompl}, we deduce that $\phi_{\mu}^* \in {\rm{PSH}}(X, \omega)$.
		Let us establish that $\phi_{\mu}^* = \phi_{\mu}$, which would then imply that $\phi_{\mu} \in {\rm{PSH}}(X, \omega)$.
		\par 
		Indeed, $\phi_{\mu}^* = \phi_{\mu}$ away from a pluripolar (or empty) subset $G \subset X$ by \cite[Theorem 4.42]{GuedjZeriahBook}.
		But since $\mu(G) = 0$, we conclude that $\phi_{\mu}^* \leq \phi$ $\mu$-almost everywhere, which means that $\phi_{\mu}^*$ is one of the contenders in the supremum from the definition of $\phi_{\mu}$, which shows that $\phi_{\mu}^* \leq \phi_{\mu}$.
		Clearly, the last inequality implies $\phi_{\mu}^* = \phi_{\mu}$, as we trivially have $\phi_{\mu}^* \geq \phi_{\mu}$.
		\par 
		Let us now establish the existence of $E \subset X$, $\mu(E) = 0$, so that $\phi_{\mu} = \phi_{K \setminus E}$.
		By (\ref{eq_phi_mu_form}), it suffices to find $E \subset X$, $\mu(E) = 0$, so that $\phi_{\mu} \leq \phi_{K \setminus E}$.
		By Choquet's lemma, cf. \cite[Lemma 4.31]{GuedjZeriahBook}, and (\ref{eq_phi_mu}) there is a countable family $E_i$, $i \in \nat$, $\mu(E_i) = 0$, so that $\phi_{\mu} = (\sup_{i \in \nat} \{ \phi_{K \setminus E_i}^* \})^*$.
		But then as before, for $F := \cup_{i \in \nat} E_i$, we have $\phi_{K \setminus E_i} \leq \phi_{K \setminus F}$ for any $i \in \nat$, and as a consequence, $\phi_{\mu} \leq \phi_{K \setminus F}^*$. 
		Remark also that we trivially have $\mu(F) = 0$.
		But then, as $\phi_{K \setminus F}^* = \phi_{K \setminus F}$ away from a pluripolar (or empty) subset $G \subset X$ by \cite[Theorem 4.42]{GuedjZeriahBook}, for $E := F \cup G$, we obtain $\phi_{K \setminus F}^* \leq \phi_{K \setminus E}$, and $\mu(E) = 0$.
		This implies that $\phi_{\mu} \leq \phi_{K \setminus E}$, which finishes the proof. 
	\end{proof}
	Let us now recall the definition of a \textit{determining measure} from \cite{SiciakDetermMeas}.
	\begin{defn}\label{defn_determ_meas}
		A non-pluripolar Borel measure $\mu$ on a compact subset $K \subset X$ is called determining for the pair $(K, h^L)$ if for each measurable subset $E \subset K$, $\mu(E) = 0$, we have $h^L_{K *} = h^L_{K \setminus E *}$.
	\end{defn}
	\begin{rem}\label{rem_determ_meas}
		a) The definition of a determining measure clearly only depends on the absolutely continuous class of the measure $\mu$.
		\par 
		b) Some authors, cf. \cite{BermanBouckBalls}, instead require that $h^L_K = h^L_{K \setminus E}$, which imposes a stronger condition, as discussed in Proposition \ref{prop_plurireg} below.
	\end{rem}

	\begin{prop}\label{prop_determ_crit}
		The following statements are equivalent: 
		\begin{enumerate}[a)]
			\item The measure $\mu$ on $K$ is determining for $(K, h^L)$.
			\item We have $h^L_{\mu} = h^L_{K *}$.
		\end{enumerate}
	\end{prop}
	\begin{proof}
		Let $E \subset K$, $\mu(E) = 0$, be given by Proposition \ref{prop_env_gen_fact}, i.e. $h^L_{\mu} = h^L_{K \setminus E}$.
		If $\mu$ is determining, we have $h^L_{K \setminus E *} = h^L_{K *}$, which shows that $a) \Rightarrow b)$, as by Proposition \ref{prop_env_gen_fact}, we have $h^L_{\mu} = h^L_{\mu *}$.
		\par  
		Now, let $F \subset K$ be a measurable subset such that $\mu(F) = 0$.
		We then immediately obtain $h^L_{\mu} \leq h^L_{K \setminus F} \leq h^L_K$, which implies that $h^L_{\mu} \leq h^L_{K \setminus F *} \leq h^L_{K *}$ and immediately shows $b) \Rightarrow a)$.
	\end{proof}
	\par 
	We stress out that there exist examples of compact subsets which are pluriregular for a fixed positive metric, but for which the Lebesgue measure is not determining, see \cite[Example 2.10 and Theorem 3.1]{SiciakExtremal}. 
	By Proposition \ref{prop_determ_crit}, this demonstrates that psh and Lebesgue envelopes do not necessarily coincide, even for such “nice" sets. 
	In particular, the distinction that we refereed to in Remark \ref{rem_thm_regul} is indeed vital for Theorem \ref{thm_distr}.
	\par 
	For the next proposition, we say that $(K, h^L)$ is \textit{$\mu$-pluriregular} if $h^L_{\mu}$ is continuous.
	\begin{prop}\label{prop_plurireg}
		The following statements are equivalent: 
		\begin{enumerate}[a)]
			\item The measure $\mu$ is determining for $(K, h^L)$ and $(K, h^L)$ is pluriregular.
			\item For each measurable subset $E \subset K$, such that $\mu(E) = 0$, we have $h^L_K = h^L_{K \setminus E}$.
			\item For an arbitrary $\psi \in {\rm{PSH}}(X, \omega)$, we have $\esssup_{\mu} (\psi - \phi) = \sup_{K}  (\psi - \phi)$.
			\item We have $h^L_K = h^L_{\mu}$.
			\item The pair $(K, h^L)$ is $\mu$-pluriregular.
		\end{enumerate}
	\end{prop}
	\begin{proof}
		The equivalence between $b)$ and $c)$ was established in \cite[Proposition 1.12]{BerBoucNys}.
		To see the implication $a) \Rightarrow b)$ remark that for a measurable subset $E \subset K$, such that $\mu(E) = 0$, the following easy chain of inequalities holds $h^L_{K \setminus E *} = h^L_{K *} = h^L_K \geq h^L_{K \setminus E}$, cf. \cite[Definition 0.2]{SiciakExtremal}, which immediately implies that $h^L_K = h^L_{K \setminus E}$ by the trivial inequality $h^L_{K \setminus E *} \leq h^L_{K \setminus E}$.
		Let us show that $b) \Rightarrow a)$.
		It is immediate that $\mu$ is determining for $(K, h^L)$, as $h^L_K = h^L_E$ implies $h^L_{K *} = h^L_{E *}$.
		The fact that $b)$ implies that $(K, h^L)$ is pluriregular was established in \cite[Proposition 1.12]{BerBoucNys}.
		\par 
		The validity of $b) \Rightarrow d)$ follows immediately from Proposition \ref{prop_env_gen_fact}.
		To discuss the converse, remark that $h^L_K$ has a lower semi-continuous potential, cf. \cite[Proposition 2.12]{SiciakExtremal}.
		The last fact follows immediately from the description 
		\begin{equation}\label{eq_h_l_k_sic}
			\frac{h^L_K(x)}{h^L(x)} = \inf \Big\{
				|s(x)|_{h^{L^{\otimes k}}}^{-\frac{1}{k}}: s \in H^0(X, L^{\otimes k}), \quad \sup_{x \in K} |s(x)|_{h^{L^{\otimes k}}} \leq 1
			\Big\},
		\end{equation}
		which is valid for psh envelopes associated with compact subsets, see \cite[Theorem 15.6]{DemPluripot}. 
		By replacing an infimum over a single section (\ref{eq_h_l_k_sic}) with an infimum over a base-point free subset of sections, we see that we can realize $h^L_K(x) / h^L(x)$ as an infimum of a set of continuous functions.
		As an infimum of a set of continuous functions is automatically upper semi-continuous, we obtain that $h^L_K$ has a lower semi-continuous potential.
		From Proposition \ref{prop_env_gen_fact}, $h^L_{\mu}$ has an upper semi-continuous potential.
		In particular, if $h^L_{\mu} = h^L_K$, then $h^L_K$ is continuous, hence $(K, h^L)$ is pluriregular.
		The fact that $\mu$ is determining if $h^L_{\mu} = h^L_K$ follows from Propositions \ref{prop_env_gen_fact} and \ref{prop_determ_crit}.
		We thus estabished that $d) \Rightarrow a)$.
		\par 
		Clearly, we also have $d) \Rightarrow e)$, as $h^L_K$ is continuous for pluriregular $(K, h^L)$.
		Let us finally show that $e) \Rightarrow d)$.
		We denote by $E \subset K$ a subset from Proposition \ref{prop_env_gen_fact}.
		Remark that $K \setminus E$ is topologically dense in $K$ as the support of $\mu$ is included in the closure of $K \setminus E$ by our standing assumption $\mu(E) = 0$.
		If $h^L_{\mu}$ is continuous, then since $K \setminus E$ is dense in $K$ and $h^L \leq h^L_{\mu}$ over $K \setminus E$, we conclude that $h^L \leq h^L_{\mu}$ over $K$.
		But then by Proposition \ref{prop_env_gen_fact}, $h^L_{\mu}$ is one of the contenders in the infimum from (\ref{defn_nonnegl_env}), which shows that $h^L_K \leq h^L_{\mu}$.
		In conjunction with the trivial bound $h^L_{\mu} \leq h^L_K$, we obtain $e) \Rightarrow d)$, which finishes the proof.
	\end{proof}
	
	We will now describe another relation between $h^L_{\mu}$ and $h^L_K$.
	Following \cite[Definition 7.24]{BouckErik21}, we introduce a regular envelope, $Q(h^L{}')$, of a bounded Hermitian metric $h^L{}'$ on $L$ as
	\begin{equation}\label{eq_reg_env}
		Q(h^L{}')
		:=
		\inf \Big\{
			h^L_0 \text{ with continuous psh potential }: h^L_0 \geq h^L{}'
		\Big\}.
	\end{equation}
	Remark that the metric $Q(h^L{}')$ has automatically a lower semi-continuous potential by the same argument as after (\ref{eq_h_l_k_sic}).
	\begin{prop}\label{prop_regul_cont_equal_extr}
		For any continuous metric $h^L$, we have $Q(h^L_{\mu}) = h^L_K$.
	\end{prop}
	
	Proposition \ref{prop_regul_cont_equal_extr} is a consequence of Proposition \ref{prop_env_gen_fact}, the topological density of $K \setminus E$ in $K$ for any $E \subset K$ verifying $\mu(E) = 0$ (which is a direct consequence of the fact that the support of $\mu$ equals $K$) and the following more general result.
	\begin{prop}\label{prop_regul_cont_equal_extr2}
		For any continuous metric $h^L$ and subset $E \subset X$, we have $Q(h^L_E) = h^L_{\overline{E}}$, where $\overline{E}$ is the closure of $E$.
	\end{prop}
	\begin{proof}
		First of all, according to (\ref{eq_h_l_k_sic}) and the discussion after, we have $Q(h^L_{\overline{E}}) = h^L_{\overline{E}}$.
		Since the application of $Q$ clearly preserves the order, we deduce immediately from the trivial bound $h^L_E \leq h^L_{\overline{E}}$ that $Q(h^L_E) \leq h^L_{\overline{E}}$.
		To finish the proof of Proposition \ref{prop_regul_cont_equal_extr}, it is hence enough to establish that if $h^L_0$ has a continuous psh potential and $h^L_0 \geq h^L_E$, then $h^L_0 \geq h^L_{\overline{E}}$.
		By the definition of $h^L_{\overline{E}}$, it suffices to establish that over ${\overline{E}}$, we have $h^L_0 \geq h^L$.
		But this follows directly from the fact that $h^L_E \geq h^L$ over $E$, the continuity of $h^L$, $h^L_0$ and the density of $E$ in ${\overline{E}}$.
	\end{proof}

	\section{Lebesgue dense points and psh envelopes}\label{sect_leb}
	This section aims to specialize the theory from Section \ref{sect_determ} to the Lebesgue envelopes.
	It turns out that a part of the previous reasoning can then be made more precise, and the set $E$ from Proposition \ref{prop_env_gen_fact} admits an explicit description through the subset of Lebesgue dense points.
	\par 
	Recall that a point $x \in X$ is called a \textit{Lebesgue dense point} of a measurable subset $A \subset X$ if 
	\begin{equation}\label{eq_leb_dens}
		\lim_{r \to 0} \frac{\lambda(B_x(r) \cap A)}{\lambda(B_x(r))} = 1,
	\end{equation}
	where $B_x(r)$ is a geodesic ball around $x$ of radius $r > 0$ with respect to some fixed Riemannian metric.
	While the notion of a geodesic ball clearly depends on the choice of the Riemannian metric, the notion of a Lebesgue dense point actually doesn't.
	To see this, remark that geodesic balls associated with different metrics have bounded eccentricity in the sense of \cite[p. 108]{SteinBook}.
	Moreover, by \cite[Corollary 3.1.7]{SteinBook}, for a sequence of open neighborhoods with bounded eccentricity $U_r$, $r > 0$, shrinking to a Lebesgue dense point $x$, as $r \to 0$, one can replace $B_x(r)$ in (\ref{eq_leb_dens}) by $U_r$.
	\par 	 
	Lebesgue's density theorem, cf. \cite[Corollary 3.1.5]{SteinBook}, states that the set $d(A)$ of Lebesgue dense points of $A$ has the full measure in $A$, i.e. $\lambda(A \setminus d(A)) = 0$.
	By applying this statement twice: for $A$ and its complement, we deduce that the symmetric difference between $A$ and $d(A)$ is Lebesgue negligible.
	\begin{prop}\label{prop_determ}
		For any measurable subset $A \subset X$, an arbitrary Kähler form $\omega$ on $X$, $\psi \in {\rm{PSH}}(X, \omega)$ and continuous $\phi : X \to \real$, the following identity holds
		\begin{equation}
			\esssup_A (\psi - \phi) 
			=
			\sup_{d(A)}  (\psi - \phi).
		\end{equation}
	\end{prop}
	\begin{proof}
		By the Lebesgue's density theorem, we clearly have 	
		\begin{equation}
			\esssup_A (\psi - \phi) 
			\leq
			\sup_{d(A)}  (\psi - \phi).
		\end{equation}
		To establish the opposite inequality, it is enough to establish that if $\psi \leq \phi$ almost everywhere on $A$, then $\psi \leq \phi$ over $d(A)$.
		\par 
		Our argument will be local, and we can assume that $\psi$ and $\phi$ are both defined on a unit ball $\mathbb{D}^n \subset \comp^n$, $\psi$ is plurisubharmonic and $\phi$ is continuous.
		Mean-value inequality then implies that for any $x \in \mathbb{D}^n$, $0 < r < 1 - |x|$, we have
		\begin{equation}\label{eq_mean_v_1}
			\psi(x) \leq \frac{\int_{B_x(r)} \psi(y) d y}{\int_{B_x(r)} d y}.
		\end{equation}
		We decompose 
		\begin{equation}\label{eq_mean_v_2}
			\frac{\int_{B_x(r)} \psi(y) d y}{\int_{B_x(r)} d y}
			=
			\frac{\int_{B_x(r) \cap A} \psi(y) d y}{\int_{B_x(r)} d y}
			+
			\frac{\int_{B_x(r) \setminus A} \psi(y) d y}{\int_{B_x(r)} d y}.
		\end{equation}
		Since $\psi$ is bounded from above over compacts (it is upper semi-continuous), for any $x \in d(A)$, the second term on the right-hand side of (\ref{eq_mean_v_2}) tends to zero, as $r \to 0$.
		By our assumption,
		\begin{equation}\label{eq_mean_v_3}
			\frac{\int_{B_x(r) \cap A} \psi(y) d y}{\int_{B_x(r)} d y}
			\leq
			\frac{\int_{B_x(r) \cap A} \phi(y) d y}{\int_{B_x(r)} d y}.
		\end{equation}
		Then for $x \in d(A)$, by the continuity of $\phi$, we see that the term on the right-hand side of (\ref{eq_mean_v_3}) tends to $\phi(x)$, as $r \to 0$.
		A combination of the above estimates yields that $\psi \leq \phi$ over $d(A)$, which finishes the proof.
	\end{proof}
	\par 
	Let us deduce the first crucial consequence of Proposition \ref{prop_determ}, establishing a connection between psh and Lebesgue envelopes. 
	\begin{prop}\label{prop_leb_env_id}
		For any measurable subset $A \subset X$ and continuous metric $h^L$ on $L$, we have $h^L_{{\rm{Leb}}, A} = h^L_{d(A)}$, and the metric $h^L_{{\rm{Leb}}, A}$ has a bounded psh potential. 
	\end{prop} 
	\begin{proof}
		At the level of potentials, the first statement reduces to an identity
		\begin{multline}\label{eq_ident_set_psh}
			\Big\{
			\psi \in {\rm{PSH}}(X, \omega) : \psi \leq \phi \text{ over } d(A)
			\Big\}
			\\
			=
			\Big\{
			\psi \in {\rm{PSH}}(X, \omega) : \psi \leq \phi \text{ almost everywhere on } A
			\Big\},
		\end{multline}
		where $\phi$ is the potential of $h^L$, calculated as in (\ref{eq_pot_metr_corr}) for a fixed positive metric $h^L_0$, verifying $\omega := 2 \pi c_1(L, h^L_0)$.
		But (\ref{eq_ident_set_psh}) is a direct consequence of Proposition \ref{prop_determ}.
		The proof of the second statement is an easy adaptation of Proposition \ref{prop_env_gen_fact}; the only additional ingredient one has to use is the Lebesgue's density theorem, i.e. that $d(A)$ has the same Lebesgue measure as $A$.
		We leave the details to the interested reader.
	\end{proof}
	
	\begin{proof}[Proof of Proposition \ref{prop_leb_env_all}]
		The first part follows from Proposition \ref{prop_leb_env_id} and the trivial fact that $h^L_X = h^L$ for any $h^L$ with a psh potential.
		The second part follows by Remark \ref{rem_leb_env_id}.
	\end{proof}

	\par 
	We finish this section with an example of a domain satisfying the assumptions of Theorem \ref{thm_regul}.
	\par 
	\begin{prop}\label{prop_example_leb_plur}
		Assume $U \subset X$ is an open subset with a $\mathscr{C}^1$-boundary. 
		Then $(\overline{U}, h^L)$ is Lebesgue pluriregular.
	\end{prop}
	\begin{proof}
		Let us show that it is sufficient to establish that $h^L_{\overline{U}} = h^L_U$.
		Indeed, by Proposition \ref{prop_leb_env_id} and the obvious fact that $\overline{U}$ is Lebesgue dense at every point of $U$, we immediately get $h^L_U \leq h^L_{{\rm{Leb}}, \overline{U}} \leq h^L_{\overline{U}}$. 
		Hence the identity $h^L_{\overline{U}} = h^L_U$ implies that $h^L_{\overline{U}} = h^L_{{\rm{Leb}}, \overline{U}}$, which according to Proposition \ref{prop_plurireg} is equivalent to the Lebesgue pluriregularity of $(\overline{U}, h^L)$.
		\par 
		Let us hence establish that $h^L_{\overline{U}} = h^L_U$.
		By the trivial bound $h^L_{\overline{U}} \geq h^L_U$, the fact that $h^L_U$ has a psh potential, cf. \cite[Proposition 9.19]{GuedjZeriahBook}, and the definition of $h^L_{\overline{U}}$, it only suffices to show that $h^L_U \geq h^L$ over $\overline{U}$.
		The proof of this fact is an easy modification of \cite[Corollary 5.3.13]{KlimekBook}, and so we will be brief.
		By the analytic accessibility criterion, cf. \cite[Proposition 5.3.12]{KlimekBook}, the set $U$ is non-thin at every point $a \in \partial \overline{U}$ in the sense of \cite[(4.8.1)]{KlimekBook} and \cite[Corollary 4.8.3]{KlimekBook}, i.e. there is an open neighborhood $V$ of $a$, such that for an arbitrary psh function $\psi$, defined in $V$, we have 
		\begin{equation}
			\limsup_{{\substack{z \to a \\ z \in U \cap V \setminus \{a\}}}} \psi(z) = \psi(a),
		\end{equation}
		opposed to the inequality where $=$ is replaced by $\leq$, which would be a trivial consequence of the upper semi-continuity of $\psi$.
		By applying this for the local potential of $h^L_U$, using the fact that $h^L$ is continuous and that over $U$, we have $h^L_U \geq h^L$, we deduce that $h^L_U(a) \geq h^L(a)$ at every $a \in \partial U$, which finishes the proof.
	\end{proof}
	Alternatively, Proposition \ref{prop_example_leb_plur} follows from Proposition \ref{prop_plurireg} and \cite[Propositions 2.17 and 2.21]{BermanPriors}, establishing that the Lebesgue measure on $\overline{U}$ is determining and that $\overline{U}$ is locally pluriregular.
	
	\section{Bernstein-Markov type properties for the Lebesgue measures}\label{sect_bm}
	The primary objective of this section is to examine different versions of the Bernstein-Markov property for the Lebesgue measures on compact subsets.
	\par 
	Recall that a positive measure $\mu$ supported on a compact subset $K \subset X$ is said to be \textit{Bernstein-Markov} with respect to $(K, h^L)$, where $h^L$ is a continuous metric on $L$, if for each $\epsilon > 0$, there is $C > 0$ such that for any $k \in \nat$, we have
	\begin{equation}\label{eq_bm_clas}
		{\textrm{Ban}}_k^{\infty}(K, h^L)
		\leq
		C
		\cdot
		\exp(\epsilon k)
		\cdot
		{\textrm{Hilb}}_k(h^L, \mu).
	\end{equation}
	\par 
	It is easy to see that the result of Tian \cite{TianBerg} on the asymptotics of Bergman kernel implies that the Bernstein-Markov property is satisfied for the Lebesgue measure on $X$, when $h^L$ is a smooth positive metric, cf. \cite[\S 2.1]{BermanBouckBalls}.
	Demailly regularization theorem, \cite{DemRegul}, \cite{GuedZeriGeomAnal}, immediately implies that it holds more generally for an arbitrary continuous metric with psh potential, cf. \cite[\S 2.1]{BermanBouckBalls}.
	\par 
	We fix a reference metric $h^L_*$, let $\omega := 2 \pi c_1(L, h^L_*)$, and denote by $\phi$ the potential of $h^L$.
	Following \cite{SiciakDetermMeas} and \cite{BerBoucNys}, we say that $\mu$ is \textit{Bernstein-Markov with respect to psh weights} if for any $\epsilon > 0$, there is $C > 0$, such that for any $\psi \in {\rm{PSH}}(X, \omega)$, $p \geq 1$, we have
	\begin{equation}\label{bm_psh_weights}
		\sup_K \big( \exp(p(\psi - \phi)) \big)
		\leq
		C
		\cdot
		\exp(\epsilon p)
		\cdot
		\int_X \exp(p(\psi - \phi)) d \mu.
	\end{equation}
	It is immediate that $\mu$ is Bernstein-Markov for  $(K, h^L)$ if it is Bernstein-Markov with respect to psh weights, since by plugging in $\psi := \frac{1}{k} \log |s|_{h^{L^{\otimes k}}_*}$ and $p = 2 k$ in (\ref{bm_weights}), we would get (\ref{eq_asympt_bm}). 
	\par 
	Let us recall the following result from \cite[Theorem 1.14]{BerBoucNys} (see also \cite[Theorems 5.1 and 5.2]{SiciakDetermMeas}).
	\begin{thm}\label{thm_bbwn_determining}
		For a compact subset $K \subset X$, a continuous metric $h^L$ on $L$ and a non-pluripolar Borel probability measure $\mu$, the following statements are equivalent: 
		\begin{enumerate}[a)]
			\item The measure $\mu$ is determining for $(K, h^L)$ and $(K, h^L)$ is pluriregular.
			\item The measure $\mu$ is Bernstein-Markov with respect to psh weights for $(K, h^L)$.
		\end{enumerate}
	\end{thm}
	\par 
	\begin{sloppypar}
	We will now explain one consequence of the above result, which will be particularly relevant in what follows.
	We say that two graded norms $N = \oplus_{k = 0}^{\infty} N_k$, $N' = \oplus_{k = 0}^{\infty} N_k'$ over $R(X, L)$ are \textit{equivalent} ($N \sim N'$) if the multiplicative gap between the graded pieces, $N_k$ and $N_k'$, is subexponential. This means that for any $\epsilon > 0$, there is $k_0 \in \nat$, such that for any $k \geq k_0$, we have
	\begin{equation}\label{defn_equiv_rel}
		\exp(-\epsilon k) \cdot N_k
		\leq
		N_k'
		\leq
		\exp(\epsilon k) \cdot N_k.
	\end{equation}
	Remark that the Bernstein-Markov condition can then be restated as the equivalence between ${\textrm{Ban}}_k^{\infty}(K, h^L)$ and ${\textrm{Hilb}}_k(h^L, \mu)$.
	\par 
	We now fix $f \in L^{\infty}(X)$, $f \neq 0$, and use the same notation $K$, $Z$, $NZ$ for the associated measurable subsets as in Introduction. 
	Let $\chi$ be an arbitrary smooth volume form on $X$. 
	\begin{cor}\label{thm_equiv_1}
		Assume that for a continuous metric $h^L$ on $L$, the pair $(K, h^L)$ is pluriregular and the Lebesgue measure on $NZ$ is determining for $(K, h^L)$.
		Then the norms $\oplus_{k = 0}^{\infty} {\textrm{Hilb}}_k(h^L, f \cdot \chi)$ and $\oplus_{k = 0}^{\infty} {\textrm{Hilb}}_k(h^L_K, \chi)$ are equivalent.
	\end{cor}
	\end{sloppypar}
	\begin{proof}
		By the Radon-Nikodym theorem, cf. \cite[Theorem 6.4.3]{SteinBook}, the measures $f d \lambda_X$ and $d \lambda_{NZ}$ are absolutely continuous with respect to each other.
		From this, Remark \ref{rem_determ_meas} and Theorem \ref{thm_bbwn_determining}, we establish that $\oplus_{k = 0}^{\infty} {\textrm{Hilb}}_k(h^L, f \cdot \chi)$ is equivalent to $\oplus_{k = 0}^{\infty} {\textrm{Ban}}_k^{\infty}(K, h^L)$.
		\par 
		By \cite[Proposition 1.8]{BermanBouckBalls} and \cite[Corollary 7.27]{BouckErik21}, for any non-pluripolar $E \subset X$, we have
		\begin{equation}\label{rem_linf_subs}
			{\textrm{Ban}}_k^{\infty}(E, h^L)
			=
			{\textrm{Ban}}_k^{\infty}(h^L_E)
			=
			{\textrm{Ban}}_k^{\infty}(Q(h^L_E)).
		\end{equation}
		\par 
		Remark that since $h^L_K$ is continuous and has a psh potential, the Bernstein-Markov property on $(X, h^L_K)$ is satisfied for the Lebesgue measure on $X$.
		From this and (\ref{rem_linf_subs}), we deduce that $\oplus_{k = 0}^{\infty} {\textrm{Ban}}_k^{\infty}(K, h^L)$ is equivalent to $\oplus_{k = 0}^{\infty} {\textrm{Hilb}}_k(h^L_K, \chi)$, which finishes the proof.
	\end{proof}
	\par 
	Since the above result doesn't hold with no pluriregularity assumption and the determining assumption on the Lebesgue measure, see Proposition \ref{thm_bbwn_determining_new} below, we need a weaker equivalence relation on the set of graded norms.
	We first define the \textit{logarithmic relative spectrum} of a norm $N_1$ on a finitely dimensional vector space $V$, $\dim V =: r$, with respect to another norm $N_2$ on $V$, as a non-increasing sequence $\lambda_j := \lambda_j(N_1, N_2) \in \real$, $j = 1, \cdots, r$, defined so that
	\begin{equation}\label{eq_log_rel_spec}
		\lambda_j
		:=
		\sup_{\substack{W \subset V \\ \dim W = j}} 
		\inf_{w \in W \setminus \{0\}} \log \frac{\| w \|_2}{\| w \|_1}.
	\end{equation}
	For $p \in [1, + \infty[$, we then let
	\begin{equation}\label{eq_dp_defn_norms}
		d_p(N_1, N_2) := \sqrt[p]{\frac{\sum_{i = 1}^{r} |\lambda_i|^p}{r}},
		\qquad 
		d_{+ \infty}(N, N') := \max \big\{ |\lambda_1|, |\lambda_r| \big\}.
	\end{equation} 
	The graded norms $N = \oplus_{k = 0}^{\infty} N_k$ and $N' = \oplus_{k = 0}^{\infty} N_k'$ on $R(X, L)$ are $p$-\textit{equivalent} ($N \sim_p N'$) if 
 	\begin{equation}\label{defn_equiv_relp}
		\frac{1}{k} d_p(N_k, N_k') \to 0, \qquad \text{as } k \to \infty.
	\end{equation}
	In \cite[\S 2.3]{FinNarSim}, we established that $\sim_p$, $p \in [1, +\infty]$, is an equivalence relation and $\sim$ equals $\sim_{+ \infty}$. 
	\par 
	The following result, which we establish in Section \ref{sect_bern_mark}, gives a weaker version of Corollary \ref{thm_equiv_1} without any assumption on the weight.
	\par 
	\begin{sloppypar}
	\begin{thm}\label{thm_equiv_2}
		For an arbitrary $f \in L^{\infty}(X)$, $f \neq 0$, and a continuous metric $h^L$ on $L$, the norms $\oplus_{k = 0}^{\infty} {\textrm{Hilb}}_k(h^L, f \cdot \chi)$ and $\oplus_{k = 0}^{\infty} {\textrm{Hilb}}_k(h^L_{{\rm{Leb}}, NZ}, \chi)$ are $p$-equivalent for any $p \in [1, +\infty[$.
		If, moreover, $h^L_{{\rm{Leb}}, NZ}$ is continuous, then the above equivalence even holds for $p = + \infty$.
		Finally, there is a positive smooth metric $h^L_0$ on $L$, so that ${\textrm{Hilb}}_k(h^L, f \cdot \chi) \geq {\textrm{Hilb}}_k(h^L_0, \chi)$ for any $k \in \nat$.
	\end{thm}
	\end{sloppypar}
	\par 
	Let us now explain the relation between Theorem \ref{thm_equiv_2} and the so-called weak Bernstein-Markov condition from \cite[p. 8]{BerBoucNys}.
	For this, on a finitely-dimensional Hermitian vector space $(V, H)$, endowed with a norm $N_0$ on $V$, we denote by ${\rm{vol}} (N)$ the volume of the unit ball of $N$ (calculated with respect to the volume element associated with $H$).
	Note that while ${\rm{vol}} (N)$ clearly depends on the choice of $H$, if we fix another norm $N_1$ on $V$, the change of variables formula implies that the difference $\log {\rm{vol}} (N_0) - \log {\rm{vol}} (N_1)$ is independent of the choice of $H$.
	\par 
	A positive non-pluripolar measure $\mu$ supported on a compact subset $K \subset X$ is said to be \textit{weakly Bernstein-Markov} with respect to $(K, h^L)$, cf. \cite[\S 3.2]{BermanBouckBalls}, if
	\begin{equation}\label{eq_weak_bm}
		\lim_{k \to \infty}
		\frac{\big|
		\log {\rm{vol}} ({\textrm{Hilb}}_k(h^L, \mu)) - \log {\rm{vol}} ({\textrm{Ban}}_k^{\infty}(K, h^L))
		\big|}{k \dim H^0(X, L^{\otimes k})}
		=
		0.
	\end{equation}
	It is straightforward that the Bernstein-Markov property implies the weak Bernstein-Markov property, as their names suggest. 
	To clarify the connection between the weak Bernstein-Markov property and Theorem \ref{thm_equiv_2}, we first establish the following result.
	\begin{prop}\label{prop_d_p_and_vol}
		For an arbitrary sequence of norms $N = \oplus_{k = 0}^{\infty} N_k$ and $N' = \oplus_{k = 0}^{\infty} N_k'$ on $R(X, L)$, we have $a) \Rightarrow b)$ in the following statements
		\begin{enumerate}[a)]
			\item We have $N \sim_1 N'$.
			\item We have \begin{equation}
				\lim_{k \to \infty}
		\frac{\big|
		\log {\rm{vol}} (N_k) - \log {\rm{vol}} (N_k')
		\big|}{k \cdot \dim H^0(X, L^{\otimes k})}
		=
		0.
			\end{equation}
		\end{enumerate}
		If we have $N_k \geq N_k'$ for any $k \in \nat$, then we also have $b) \Rightarrow a)$. 
		If there is $c > 0$, so that for any $k \in \nat^*$, we have $\exp(-c k) \cdot N_k' \leq N_k \leq \exp(c k) \cdot N_k'$, then the condition $a)$ above is equivalent to
		\begin{enumerate}[a)]
			\setcounter{enumi}{2}
			\item We have $N \sim_p N'$ for any $p \in [1, +\infty[$.
		\end{enumerate}
	\end{prop}
	\begin{proof}
		Before all, let us explain that it suffices to consider Hermitian norms in the above statements.
		To see this, remark that John ellipsoid theorem, cf. \cite[\S 3]{PisierBook}, says that for any normed vector space $(V, N_V)$, there is a Hermitian norm $H_V$ on $V$, verifying 
		\begin{equation}\label{eq_john_ellips}
			H_V \leq N_V \leq \sqrt{\dim V} \cdot H_V.
		\end{equation}
		While for Hermitian norms, $d_p$ is a distance (and in particular satisfies the triangle inequality), for arbitrary norms $N_0, N_1, N_2$ on $V$, a weak version of this property holds, as discussed in \cite[(3.9)]{FinTits}:
		\begin{equation}\label{eq_tr_weak}
			d_p(N_0, N_2) \leq d_p(N_0, N_1) + d_p(N_1, N_2) + \log \dim V.
		\end{equation}
		\par 
		If we now denote by $H = \oplus_{k = 0}^{\infty} H_k$ and $H' = \oplus_{k = 0}^{\infty} H_k'$ the Hermitian norms on $R(X, L)$ associated by the John ellipsoid theorem with $N = \oplus_{k = 0}^{\infty} N_k$ and $N' = \oplus_{k = 0}^{\infty} N_k'$ respectively, then from (\ref{eq_tr_weak}), we immediately deduce
		\begin{equation}
			\Big| d_p(N_k, N_k') - d_p(H_k, H_k') \Big| \leq 3 \log n_k,
		\end{equation}
		where $n_k := \dim H^0(X, L^{\otimes k})$.
		Similarly, we have
		\begin{equation}
		\big|
			\log {\rm{vol}} (N_k) - \log {\rm{vol}} (H_k)
			\big|
			\leq
			2 \log n_k,
			\qquad \big|
			\log {\rm{vol}} (N_k') - \log {\rm{vol}} (H_k')
			\big|
			\leq
			2 \log n_k.
		\end{equation}
		As $\lim_{k \to \infty} \frac{1}{k} \log n_k = 0$, we see that in both conditions $a)$ and $b)$, we can safely change $N_k$ to $H_k$ and $N_k'$ to $H_k'$.
		Moreover, it is easy to see that if there is $c > 0$, so that for any $k \in \nat^*$, we have $\exp(-c k) \cdot N_k' \leq N_k \leq \exp(c k) \cdot N_k'$, then there is $c > 0$, so that for any $k \in \nat^*$, we have $\exp(-c k) \cdot H_k' \leq H_k \leq \exp(c k) \cdot H_k'$; and if $N_k \geq N_k'$, then we could easily arrange $H_k \geq H_k'$ by multiplying $H_k$ by $n_k$, which again would not affect neither of the conditions $a)$ and $b)$.
		Hence, we may assume without loss of generality that $N_k$ and $N_k'$ are Hermitian.
		\par 
		It is then immediate that for the logarithmic relative spectrum, $\lambda_i^k$, $i = 1, \ldots, n_k$, of $N_k$ with respect to $N_k'$, we have
		\begin{equation}\label{eq_rel_vol_spec}
			\log {\rm{vol}} (N_k') - \log {\rm{vol}} (N_k)
			=
			\sum_{i=1}^{n_k} \lambda_i^k.
		\end{equation}
		By this and (\ref{eq_dp_defn_norms}), we see immediately that $a) \Rightarrow b)$.
		We also see that as if $N_k \geq N_k'$, then $\lambda_i$ are all positive, and we have
		\begin{equation}\label{eq_vol_d1}
			\frac{\log {\rm{vol}} (N_k') - \log {\rm{vol}} (N_k)}{n_k}
			=
			d_1(N_k, N_k'),
		\end{equation}
		which clearly shows that $b) \Rightarrow a)$.
		\par 
		The condition $\exp(-c k) \cdot N_k' \leq N_k \leq \exp(c k) \cdot N_k'$ assures that $|\lambda_i^k| \leq ck$ for any $k \in \nat^*$, $i = 1, \ldots, n_k$.
		It is then immediate from (\ref{eq_dp_defn_norms}) that $a) \Rightarrow c)$.
		The inverse implication is trivial.
	\end{proof}
	
	Compare the following result with Proposition \ref{prop_determ_crit}.
	\begin{prop}\label{thm_bbwn_determining_new}
		For a compact subset $K \subset X$ and a continuous metric $h^L$ on $L$, the following statements are equivalent:
		\begin{enumerate}[a)]
			\item We have $h^L_{{\rm{Leb}}, K} = h^L_{K *}$.
			\item The Lebesgue measure on $K$ is weakly Bernstein-Markov with respect to $(K, h^L)$.
		\end{enumerate}
		If, moreover, $K$ is pluriregular, then the following statements are equivalent:
		\begin{enumerate}[a)]
			\setcounter{enumi}{2}
			\item We have $h^L_{{\rm{Leb}}, K} = h^L_K$.
			\item The Lebesgue measure on $K$ is Bernstein-Markov with respect to $(K, h^L)$.
		\end{enumerate}
	\end{prop}
	\begin{rem}\label{rem_totik}
		The analogue of this statement for arbitrary non-pluripolar Borel measures doesn't hold.
		Indeed, V. Totik in \cite[Example 8.1]{BernsteinMarkovSurvey} gave an example of a non-pluripolar Borel measure $\mu$ satisfying Bernstein-Markov property, but which is not determining.
		According to Proposition \ref{prop_determ_crit}, $h^L_{\mu} \neq h^L_{K *}$, and so the analogues of the implications $d) \Rightarrow c)$ and $b) \Rightarrow a)$ do not hold.
		\par 
		Nevertheless, Proposition \ref{prop_plurireg} along with Theorem \ref{thm_bbwn_determining} imply that for an arbitrary non-pluripolar Borel measure, the implication $c) \Rightarrow d)$ holds.
		Also, in Proposition \ref{prop_weak_bm}, we establish that the analogue of $a) \Rightarrow b)$ holds for an arbitrary non-pluripolar Borel measure.
	\end{rem}
	\begin{proof} 
		By the discussion after (\ref{eq_h_l_k_sic}), there is a sequence of continuous metrics with psh potentials which decay towards $h^L_K$ almost everywhere.
		In particular, $h^L_K$ is \textit{regularizable from above} in the notations of \cite{BedfordRegulBelow}, \cite{BouckErik21}, \cite{FinNarSim}.
		Hence, by \cite[Corollary 2.21]{FinNarSim}, the norm $\oplus_{k = 0}^{+\infty} {\textrm{Ban}}_k^{\infty}(h^L_K)$ is $p$-equivalent to $\oplus_{k = 0}^{+\infty} {\textrm{Ban}}_k^{\infty}(h^L_{K *})$ for any $p \in [1, +\infty[$, and by \cite[Proposition 2.18]{FinNarSim}, the norm $\oplus_{k = 0}^{+\infty} {\textrm{Ban}}_k^{\infty}(h^L_{K *})$ is $p$-equivalent to $\oplus_{k = 0}^{+\infty} {\textrm{Hilb}}_k(h^L_{K *}, \chi)$ for any $p \in [1, +\infty[$.
		From all this and (\ref{rem_linf_subs}), we deduce
		\begin{equation}\label{eq_ban_hilb_etc}
			\oplus_{k = 0}^{+ \infty} {\textrm{Hilb}}_k(h^L_{K *}, \chi) \sim_p \oplus_{k = 0}^{+ \infty} {\textrm{Ban}}_k^{\infty}(K, h^L), \quad \text{for any $p \in [1, +\infty[$}.
		\end{equation}
		\par 
		By the result of Darvas-Lu-Rubinstein \cite[Theorem 1.2]{DarvLuRub} (refining previous results of Chen-Sun \cite{ChenSunQuant} and Berndtsson \cite[Theorem 3.3]{BerndtProb}), for any two Hermitian metrics $h^L_0$, $h^L_1$ with bounded psh potentials, we have
		\begin{equation}\label{eq_berndt}
			\lim_{k \to \infty} \frac{1}{k} d_p \Big({\textrm{Hilb}}_k(h^L_0, \chi), {\textrm{Hilb}}_k(h^L_1, \chi) \Big)
			=
			d_p(h^L_0, h^L_1),
		\end{equation}
		where $d_p(h^L_0, h^L_1)$ is the Darvas distance between $h^L_0$ and $h^L_1$, \cite{DarvasFinEnerg}, which is a generalization of Mabuchi distance \cite{Mabuchi}, corresponding to $p = 2$, see Section \ref{sect_mabuch} for more details.
		\par 
		\begin{sloppypar}
		Directly from this, Theorem \ref{thm_equiv_2}, (\ref{eq_ban_hilb_etc}) and (\ref{eq_berndt}), we see 
		\begin{equation}\label{eq_berndt101}
			\lim_{k \to \infty} \frac{1}{k} d_p \Big({\textrm{Ban}}_k^{\infty}(K, h^L), {\textrm{Hilb}}_k(K, h^L) \Big)
			=
			d_p(h^L_{K *}, h^L_{{\rm{Leb}}, K}).
		\end{equation}
		Since $d_p$ separates points, we see that $\oplus_{k = 0}^{+ \infty} {\textrm{Ban}}_k^{\infty}(K, h^L) \sim_p \oplus_{k = 0}^{+ \infty} {\textrm{Hilb}}_k(K, h^L)$ for any $p \in [1, +\infty[$ if and only if $h^L_{K *} = h^L_{{\rm{Leb}}, K}$.
		Note, however, that we trivially have ${\textrm{Ban}}_k^{\infty}(K, h^L) \geq {\textrm{Hilb}}_k(K, h^L)$, and so by Theorem \ref{thm_equiv_2} and Proposition \ref{prop_d_p_and_vol}, we have $\oplus_{k = 0}^{+ \infty} {\textrm{Ban}}_k^{\infty}(K, h^L) \sim_p \oplus_{k = 0}^{+ \infty} {\textrm{Hilb}}_k(K, h^L)$ if and only if the Lebesgue measure on $K$ is weakly Bernstein-Markov with respect to $(K, h^L)$.
		We thus get the equivalence between $a)$ and $b)$.
		\end{sloppypar}
		\par 
		We will now assume that $K$ is pluriregular.
		Then $h^L_{K *} = h^L_K$, and so $b \Rightarrow a)$, we conclude immediately that $d) \Rightarrow c)$, as $d) \Rightarrow b)$.
		The inverse implication $c) \Rightarrow d)$ follows from Proposition \ref{prop_plurireg} and Theorem \ref{thm_bbwn_determining}.
	\end{proof}
	As one application, we give an answer to \cite[Question 1.11]{BerBoucNys} for the Lebegue measures.
	\begin{cor}
		The Lebesgue measure on a Lebesgue non-negligible compact subset $K \subset X$ is weakly Bernstein-Markov on $(K, h^L)$ for a continuous metric $h^L$ if and only if it is Bernstein-Markov with respect to psh weights on $(K, h^L)$.
	\end{cor}
	\begin{rem}
		As observed in \cite[\S 8]{BernsteinMarkovSurvey}, the analogue of this statement for an arbitrary non-pluripolar Borel measure doesn't hold.
	\end{rem}
	\begin{proof}
		Follows from Theorem \ref{thm_bbwn_determining} and Propositions \ref{prop_plurireg}, \ref{thm_bbwn_determining_new}.
	\end{proof}
	
	\section{From Toeplitz to Transfer operators and back}\label{sect_transf}
	The main goal of this section is to establish that the logarithm of a Toeplitz operator is asymptotically Toeplitz, i.e. to prove Theorem \ref{thm_log2}.
	The major idea behind the proof is to first interpret the logarithm of a Toeplitz operator as a transfer map between $L^2$-norms associated with different weights, and then to proceed by the comparison of the $L^2$-norms associated with weights to $L^2$-norms associated with smooth volume forms, the transfer operators between which has been studied recently in \cite{FinSubmToepl}. 
	\par 
	Let us recall the definition of the transfer map first.
	Let $V$ be a complex vector space and $H_0, H_1$ be two Hermitian norms on $V$. The \textit{transfer map}, $T \in {\rm{End}}(V)$, between $H_0, H_1$, is the Hermitian operator (with respect to both $H_0, H_1$), defined so that the Hermitian products $\langle \cdot, \cdot \rangle_{H_0}, \langle \cdot, \cdot \rangle_{H_1}$ induced by $H_0$ and $H_1$, are related as $\langle \cdot, \cdot \rangle_{H_1} = \langle \exp(-T) \cdot, \cdot \rangle_{H_0}$.
	\par 
	A crucial point is that the asymptotic class of the transfer map doesn't depend on the equivalence class of the norm.
	More specifically, consider a sequence of Hermitian norms $H_k^i$, $i = 0, 1, 2$ on $H^0(X, L^{\otimes k})$, $k \in \nat$. 
	We denote by $T_k^j \in {\rm{End}}(H^0(X, L^{\otimes k}))$, $j = 1, 2$, the transfer maps between $H_k^0$ and $H_k^j$, $k \in \nat$.
	The following proposition will be crucial in our approach.
	\begin{prop}\label{prop_equiv_tranf}
		Assume that for some $p \in [1, +\infty]$, we have $\oplus_{k = 0}^{\infty} H_k^1 \sim_p \oplus_{k = 0}^{\infty} H_k^2$.
		Then for any $\epsilon > 0$, there is $k_0 \in \nat$, such that for any $k \geq k_0$, we have $
			\| 
				T_k^1 - T_k^2
			\|_p
			\leq
			\epsilon k$.
	\end{prop}
	\begin{proof}
		Let us first establish the statement for $p \in [1, +\infty[$.
		Recall that the exponential metric increasing property, as stated in \cite[Theorem 6.1.4 and (6.36)]{BhatiaBook}, asserts that for any $p \in [1, +\infty[$,
		\begin{equation}\label{eq_emi}
			\big\| 
				T_k^1 - T_k^2
			\big\|_p
			\leq
			d_p(H_k^1, H_k^2),
		\end{equation}
		which immediately implies Proposition \ref{prop_equiv_tranf} for $p \in [1, +\infty[$.
		To establish Proposition \ref{prop_equiv_tranf} for $p = +\infty$, it then suffices to take the limit $p \to +\infty$ in (\ref{eq_emi}).
		\begin{comment}
		By our assumption, for any $\epsilon > 0$, there is $k_0 \in \nat$, such that for any $k \geq k_0$, we have 
		\begin{equation}
			H_k^1 \leq H_k^2 \cdot \exp(\epsilon k), \qquad H_k^2 \leq H_k^1 \cdot \exp(\epsilon k).
		\end{equation}
		Hence, we have 
		\begin{equation}
			\exp(-T_k^1) \leq \exp(-T_k^2) \cdot \exp(\epsilon k), \qquad \exp(-T_k^2) \leq \exp(-T_k^1) \cdot \exp(\epsilon k), 
		\end{equation}
		where the order $A_k \leq B_k$ for self-adjoint $A_k, B_k \in {\enmr{H^0(X, L^{\otimes k})}}$ is the Loewner order with respect to $H_k^0$.
		Proposition \ref{prop_equiv_tranf} for $p = +\infty$ now follows from the above estimate and the fact that $\log$ is order preserving, i.e. if $0 \leq A_k \leq B_k$, then $\log (A_k) \leq \log (B_k)$, cf. \cite{LoewnerArt} or \cite{SimonBarryLoewn}.
		\end{comment}
	\end{proof}
	\par 
	To make use of the above statement, we need to be able to calculate the transfer maps between various $L^2$-norms.
	We will fix a positive smooth metric $h^L_0$ on $L$, and a metric $h^L_1$ on $L$ with bounded psh potential.
	Fix also two smooth volume forms $\chi_i$, $i = 0, 1$, on $X$.
	We denote by $T_k(h^L_0, h^L_1) \in {\rm{End}}(H^0(X, L^{\otimes k}))$ the transfer map between ${\textrm{Hilb}}_k(h^L_0, \chi_0)$ and ${\textrm{Hilb}}_k(h^L_1, \chi_1)$.
	Let $h^L_t$, $t \in [0, 1]$, be the Mabuchi geodesic between $h^L_0$ and $h^L_1$.
	We define $\phi(h^L_0, h^L_1)$ as in (\ref{eq_speed_mab_geod}).
	\begin{thm}\label{thm_trasnfer}
		The sequence $\frac{1}{k} T_k(h^L_0, h^L_1)$, $k \in \nat^*$, forms an asymptotically Toeplitz operator of Schatten class with symbol $\phi(h^L_0, h^L_1)$.
		If, moreover, $h^L_1$ is continuous, then $\frac{1}{k} T_k(h^L_0, h^L_1)$, $k \in \nat^*$, even forms an asymptotically Toeplitz operator with symbol $\phi(h^L_0, h^L_1)$.
	\end{thm}
	\begin{rem}\label{rem_trasnfer}
		The second part of Theorem \ref{thm_trasnfer} was established by the author in \cite[Theorem 3.2]{FinSubmToepl}. 
	\end{rem}
	\begin{proof}
		By Remark \ref{rem_trasnfer}, it suffices to establish the first part of Theorem \ref{thm_trasnfer}, on which we concentrate from now on. 
		By the Demailly regularization theorem, there is a sequence of positive metrics $h^{L, i}_1$, $i \in \nat$, on $L$, increasing towards $h^L_1$, see \cite{DemRegul}, \cite{GuedZeriGeomAnal}.
		\par 
		By \cite[Proposition 4.6]{DarvasFinEnerg}, for the Darvas distance, $d_p$, $p \in [1, +\infty[$, we have
		\begin{equation}\label{eq_dist_conv}
			\lim_{i \to 0} d_p(h^L_1, h^{L, i}_1) = 0.
		\end{equation}
		From (\ref{eq_emi}), (\ref{eq_berndt}) and (\ref{eq_dist_conv}), we conclude that for any $\epsilon > 0$, there is $i_0 \in \nat$, such that for any $i \geq i_0$, there is $k_0 \in \nat$, such that for any $k \geq k_0$, we have 
		\begin{equation}\label{eq_trans_1}
			\Big\|
				T_k(h^L_0, h^{L, i}_1)
				-
				T_k(h^L_0, h^L_1)
			\Big\|_p
			\leq
			\epsilon k.
		\end{equation}
		\par 
		Now, from Darvas-Lu \cite[Theorem 3.1]{DarLuGeod} and (\ref{eq_dist_conv}), we deduce that 
		\begin{equation}
			\phi(h^L_0, h^{L, i}_1) \quad \text{converges almost everywhere to} \quad \phi(h^L_0, h^{L, i}_1), \text{ as $i \to \infty$}.
		\end{equation}
		From this and \cite[the proof of Proposition 3.5]{FinSubmToepl}, we conclude that for any $\epsilon > 0$, there is $i_0 \in \nat$, such that for any $i \geq i_0$, there is $k_0 \in \nat$, such that for any $k \geq k_0$, we have 
		\begin{equation}\label{eq_trans_2}
			\Big\|
				T_k(\phi(h^L_0, h^{L, i}_1))
				-
				T_k(\phi(h^L_0, h^L_1))
			\Big\|_p
			\leq
			\epsilon.
		\end{equation}
		Theorem \ref{thm_trasnfer} for $p \in [1, +\infty[$, now follows directly from the validity of Theorem \ref{thm_trasnfer} for $h^L_0 := h^L_0$ and $h^L_1 = h^{L, i}_1$, (\ref{eq_trans_1}) and (\ref{eq_trans_2}).
	\end{proof}
	
	\begin{proof}[Proof of Theorem \ref{thm_log2}]
		The key observation for our proof is that $- \log (T_k(f))$ can be interpreted as a transfer map between ${\textrm{Hilb}}_k(h^L, \chi)$ and ${\textrm{Hilb}}_k(h^L, f \cdot \chi)$.
		This immediately follows from the identity $\scal{T_k(f) s_1}{s_2}_{{\textrm{Hilb}}_k(h^L, \chi)} = \scal{s_1}{s_2}_{{\textrm{Hilb}}_k(h^L, f \cdot \chi)}$, where $s_1, s_2 \in H^0(X, L^{\otimes k})$.
		By this and Theorem \ref{thm_equiv_2}, we deduce
		\begin{equation}
			T_k(f) \geq \exp(- T_k(h^L, h^L_0)),
		\end{equation}
		where $h^L_0$ is as in Theorem \ref{thm_equiv_2}.
		By using the fact that $\log$ is order preserving, i.e. if $0 \leq A_k \leq B_k$, then $\log (A_k) \leq \log (B_k)$, cf. \cite{LoewnerArt} or \cite{SimonBarryLoewn}, we get $- \log (T_k(f)) \leq T_k(h^L, h^L_0)$.
		By this, Theorem \ref{thm_trasnfer} and the trivial bound $T_k(f) \leq \esssup_X f \cdot {\rm{Id}}$, there is $C > 0$, such that for any $k \in \nat^*$, we have $\|  \log (T_k(f)) \| \leq C k$.
		Note that this estimate provides an independent proof of (\ref{eq_min_bndliumar}).
		\par 
		Similarly, by Theorems \ref{thm_equiv_2} and Proposition \ref{prop_equiv_tranf}, we deduce that for any $\epsilon > 0$, $p \in [1, +\infty[$, there is $k_0 \in \nat$, such that for any $k \geq k_0$, we have
		\begin{equation}\label{eq_thm_log2pf}
			\Big\|
				\log (T_k(f))
				+
				T_k(h^L, h^L_{{\rm{Leb}}, NZ})
			\Big\|_p
			\leq
			\epsilon k.
		\end{equation}
		By Theorem \ref{thm_trasnfer} and (\ref{eq_thm_log2pf}), we then deduce that $\frac{1}{k} \log (T_k(f))$, $k \in \nat^*$, forms an asymptotically Toeplitz operator of Schatten class with symbol $\phi(h^L, NZ)$.
		The last statement of Theorem \ref{thm_log2} follows from Corollary \ref{thm_equiv_1}, Theorem \ref{thm_trasnfer} and Propositions \ref{prop_plurireg}, \ref{prop_equiv_tranf} in a similar way.
	\end{proof}

	\section{Asymptotic class of the weighted $L^2$-norms}\label{sect_bern_mark}
	The main goal of this section is to determine the asymptotic class of the $L^2$-norm associated with a weight, i.e. to establish Theorem \ref{thm_equiv_2}.
	The proof is based on the following more general result.
	\begin{thm}\label{thm_asympt_bm}
		We fix a non-pluripolar Borel probability measure $\mu$ on $X$, and a continuous metric $h^L_0$ on $L$, verifying $h^L_{\mu} \geq h^L_0$, where $h^L_{\mu}$ is the non-negligible psh envelope defined in (\ref{defn_nonnegl_env}).
		Then for any $\epsilon > 0$, there is $C > 0$, such that for any $k \in \nat$, we have
		\begin{equation}\label{eq_asympt_bm}
			{\textrm{Hilb}}_k(h^L, \mu)
			\geq
			C
			\cdot
			\exp(- \epsilon k)
			\cdot
			{\textrm{Ban}}_k^{\infty}(X, h^L_0).
		\end{equation}
	\end{thm}
	\begin{proof}
		Our proof builds on the proof of Theorem \ref{thm_bbwn_determining} from \cite{BerBoucNys}, but diverges from it in one essential aspect: to overcome the use of determining assumption on the measure from Theorem \ref{thm_bbwn_determining}, we use non-negligible envelopes instead of the psh envelopes.
		\par 
		To explain this in details, we work on the level of potentials instead of metrics. 
		We fix a reference metric $h^L_*$, denote $\omega := 2 \pi c_1(L, h^L_*)$, and denote by $\phi$, $\phi_0$ and $\phi_{\mu}$ the potentials of $h^L$, $h^L_0$ and $h^L_{\mu}$ respectively, see (\ref{eq_pot_metr_corr}).
		We shall establish the following result: for any $\epsilon > 0$, there is $C > 0$, such that for any $\psi \in {\rm{PSH}}(X, \omega)$, $p \geq 1$, we have
		\begin{equation}\label{bm_weights}
			\sup_X \big( \exp(p(\psi - \phi_0)) \big)
			\leq
			C
			\cdot
			\exp(\epsilon p)
			\cdot
			\int_X \exp(p(\psi - \phi)) d \mu,
		\end{equation}
		which we suggest to compare (\ref{bm_psh_weights}).
		Once we establish (\ref{bm_weights}), by plugging in $\psi := \frac{1}{k} \log |s|_{h^{L^{\otimes k}}_*}$ and $p = 2 k$ in (\ref{bm_weights}), we would get (\ref{eq_asympt_bm}).
		For $p \in [1, +\infty[$, we introduce the functionals
		\begin{equation}
			F_p(\psi) := \frac{1}{p} \log \int_X \exp(p(\psi - \phi)) d \mu, 
			\qquad 
			F(\psi) := \sup_X (\psi - \phi_{\mu}),
		\end{equation}
		defined for $\psi \in {\rm{PSH}}(X, \omega)$.
		\par 
		\begin{sloppypar}
		Now, by the usual facts from integration theory, as $p \to + \infty$, $F_p$ converges pointwise on ${\rm{PSH}}(X, \omega)$ to $\log \| \exp(\psi - \phi) \|_{L^{\infty}(\mu)}$.
		Clearly, we have $\log \| \exp(\psi - \phi) \|_{L^{\infty}(\mu)} = \esssupp_{\mu} (\psi-\phi)$.
		By (\ref{eq_esssup_phi_mu_ident}), we then deduce that $\esssupp_{\mu} (\psi-\phi) = F(\psi)$.
		In conclusion, we obtain that as $p \to + \infty$, $F_p$ converges pointwise on ${\rm{PSH}}(X, \omega)$ to $F$.
		\end{sloppypar}
		\par 
		By \cite[Lemma 1.14]{BerBoucNys}, the functionals $F_p$ are continuous on ${\rm{PSH}}(X, \omega)$.
		Moreover, $F_p$ are increasing in $p \in [1, +\infty[$ by \cite[proof of Theorem 1.13]{BerBoucNys}.
		Also, $F - F_p$ is clearly invariant by translation (adding a constant to the parameter), thus it descends on a function on the quotient space, which is isomorphic with the space of positive $(1, 1)$-currents lying in the cohomology class of $c_1(L)$, which we denote by $\mathcal{T}(X, \omega)$.
		The latter space is compact (in the weak topology of currents), cf. \cite[Proposition 8.5]{GuedjZeriahBook}.
		\par 
		Unfortunately, the functional $F$ is not necessarily upper semi-continuous (as Hartogs lemma, cf. \cite[Theorem 1.46]{GuedjZeriahBook}, doesn't apply for the non-necessarily continuous potential $\phi_{\mu}$).  
		The assumption of pluriregularity of $(K, h^L)$ in the proof of \cite[Theorem 1.13]{BerBoucNys} enters precisely at this crucial moment.
		To bypass additional assumptions, we consider the functional $G(\psi) := \sup_X (\psi - \phi_0)$, $\psi \in {\rm{PSH}}(X, \omega)$, for which we clearly have $F \geq G$, and which is upper semi-continuous by the Hartogs lemma (we use here crucially our assumption that $\phi_0$ is continuous).
		Then the functionals $G_p := \max(G - F_p, 0)$ on $\mathcal{T}(X, \omega)$ are upper semi-continuous, they are decreasing in $p \in [1, + \infty[$, and converging, as $p \to + \infty$, to zero. 
		By the already mentioned compactness of $\mathcal{T}(X, \omega)$ and Dini's theorem, we conclude that the convergence of $G_p$ to $0$, as $p \to + \infty$, is uniform on ${\rm{PSH}}(X, \omega)$.
		The reader will check that the last fact is just a reformulation of (\ref{bm_weights}), with only one change: the inequality holds for $p \geq p_0$, where $p_0$ depends on $\epsilon$, and for $C = 1$.
		\par 
		Now, by the compactness of $\mathcal{T}(X, \omega)$ and the fact that $G_p$ are upper semi-continuous and decreasing in $p$, there is $C > 0$, so that $C \geq G_1 \geq G_p$ for any $p \geq 1$.
		This fact corresponds to (\ref{bm_weights}), with only one change: the inequality holds for $p \geq 1$, but $\epsilon$ is replaced by $C$.
		A combination of this and a previously established estimate yields (\ref{bm_weights}) in full generality.
	\end{proof}

	\begin{proof}[Proof of Theorem \ref{thm_equiv_2}]
		Remark first that directly from Proposition \ref{prop_leb_env_id} and Lebesgue's density theorem, for $C := \esssup_K f$, we have
		\begin{equation}\label{eq_bnd_hilb_leb}
			C 
			\cdot
			{\textrm{Hilb}}_k(h^L_{{\rm{Leb}}, NZ}, \chi)
			\geq
			{\textrm{Hilb}}_k(h^L, f \cdot \chi).
		\end{equation}
		The existence of $C > 0$, so that the lower bound ${\textrm{Hilb}}_k(h^L, f \cdot \chi) \geq C \cdot \exp(-\epsilon k) \cdot {\textrm{Hilb}}_k(h^L_0, \chi)$, holds for a fixed continuous metric $h^L_0$ on $L$, verifying $h^L_{{\rm{Leb}}, NZ} \geq h^L_0$, follows immediately from Theorem \ref{thm_asympt_bm}.
		The statement concerning the $\sim_{+ \infty}$-equivalence follows also immediately from Theorem \ref{thm_asympt_bm}, as under the continuity of $h^L_{{\rm{Leb}}, NZ}$, one can take $h^L_0 := h^L_{{\rm{Leb}}, NZ}$.
		It is hence only left to establish the statement for the $\sim_p$-equivalence, $p \in [1, +\infty[$.
		\par 
		Recall that Lidskii inequality implies that if the Hermitian metrics $H_0, H_1, H_2$ on a finitely dimensional vector space $V$ are ordered as $H_0 \leq H_1 \leq H_2$, then $d_p(H_0, H_1) \leq d_p(H_0, H_2)$, cf. \cite[Theorem 2.7]{DarvLuRub}.
		From this, Theorem \ref{thm_asympt_bm} and (\ref{eq_bnd_hilb_leb}), we deduce that for any $\epsilon > 0$, and continuous $h^L_0$, verifying $h^L_{{\rm{Leb}}, NZ} \geq h^L_0$, there is $k_0 \in \nat$, such that for any $k \geq k_0$, we have
		\begin{equation}\label{eq_dp_hilb_lebenv}
			d_p
			\Big(
				{\textrm{Hilb}}_k(h^L_{{\rm{Leb}}, NZ}, \chi), {\textrm{Hilb}}_k(h^L, f \cdot \chi)
			\Big)
			\leq
			\epsilon k
			+
			d_p
			\Big(
				{\textrm{Hilb}}_k(h^L_{{\rm{Leb}}, NZ}, \chi), {\textrm{Hilb}}_k(h^L_0, \chi)
			\Big).
		\end{equation}
		Now, by Demailly regularization theorem, there is a sequence of positive metrics $h^L_i$, $i \in \nat^*$, on $L$, increasing towards $h^L_{{\rm{Leb}}, NZ}$.
		If we apply (\ref{eq_berndt}) and (\ref{eq_dp_hilb_lebenv}) for $h^L_0 := h^L_i$, we would get
		\begin{equation}\label{eq_dp_hilb_lebenv2}
			\limsup_{k \to \infty} \frac{1}{k}  d_p
			\Big(
				{\textrm{Hilb}}_k(h^L_{{\rm{Leb}}, NZ}, \chi), {\textrm{Hilb}}_k(h^L, f \cdot \chi)
			\Big)
			\leq
			\epsilon + 
			d_p(h^L_{{\rm{Leb}}, NZ}, h^L_i).
		\end{equation}
		From (\ref{eq_dist_conv}), by taking the limit $i \to +\infty$ in (\ref{eq_dp_hilb_lebenv2}), we deduce Theorem \ref{thm_equiv_2}.
	\end{proof}
	\par 
	Remark that by Proposition \ref{prop_env_gen_fact} and Theorem \ref{thm_asympt_bm}, for an arbitrary non-pluripolar Borel measure $\mu$ and a continuous Hermitian metric $h^L_0$ on $L$, so that $h^L_{\mu} \geq h^L_0$, the following holds. 
	For any $\epsilon > 0$, there is $C > 0$, such that for any $k \in \nat$, we have
	\begin{equation}\label{eq_bound_below_measure}
		{\textrm{Hilb}}_k(h^L_{\mu}, \mu) 
		\geq 
		{\textrm{Hilb}}_k(h^L, \mu)
		\geq
		C
		\cdot
		\exp(- \epsilon k)
		\cdot
		{\textrm{Hilb}}_k(h^L_0, \mu).
	\end{equation}
	However, it is not generally true that $\oplus_{k = 0}^{+ \infty} {\textrm{Hilb}}_k(h^L, \mu) \sim_p \oplus_{k = 0}^{+ \infty} {\textrm{Hilb}}_k(h^L_{\mu}, \chi)$, for some $p \in [1, +\infty[$.
	To illustrate this, we need the following result.
	\begin{prop}\label{prop_weak_bm}
		Assume that $h^L_{\mu} = h^L_{K *}$, then the measure $\mu$ is weakly Bernstein-Markov with respect to $(K, h^L)$.
		Moreover, if $\mu$ is such that $\oplus_{k = 0}^{+ \infty} {\textrm{Hilb}}_k(h^L, \mu) \sim_p \oplus_{k = 0}^{+ \infty} {\textrm{Hilb}}_k(h^L_{\mu}, \chi)$ for some $p \in [1, +\infty[$, then $h^L_{\mu} = h^L_{K *}$ if and only if $\mu$ is weakly Bernstein-Markov with respect to $(K, h^L)$.
	\end{prop}
	Immediately from Remark \ref{rem_totik} and Proposition \ref{prop_weak_bm}, we see that there are non-pluripolar Borel measures $\mu$, for which $\oplus_{k = 0}^{+ \infty} {\textrm{Hilb}}_k(h^L, \mu) \not\sim_p \oplus_{k = 0}^{+ \infty} {\textrm{Hilb}}_k(h^L_{\mu}, \chi)$ for any $p \in [1, +\infty[$.
	\begin{sloppypar}
	\begin{proof}[Proof of Proposition \ref{prop_weak_bm}]
		Remark that ${\textrm{Ban}}_k^{\infty}(K, h^L) \geq {\textrm{Hilb}}_k(h^L, \mu)$, and hence for any continuous metric $h^L_0$ so that $h^L_{\mu} \geq h^L_0$, as in (\ref{eq_dp_hilb_lebenv}), we have
		\begin{equation}\label{eq_dp_hilb_lebenv_modif}
			d_p
			\Big(
				{\textrm{Hilb}}_k(h^L_{\mu}, \chi), {\textrm{Hilb}}_k(h^L, \mu)
			\Big)
			\leq
			\epsilon
			k
			+
			d_p
			\Big(
				{\textrm{Ban}}_k^{\infty}(K, h^L), {\textrm{Hilb}}_k(h^L_0, \chi)
			\Big).
		\end{equation}
		From (\ref{eq_ban_hilb_etc}) and (\ref{eq_dp_hilb_lebenv_modif}), we then conclude in the same way as in (\ref{eq_dp_hilb_lebenv2}) that 
		\begin{equation}\label{eq_dp_hilb_lebenv2_modif}
			\limsup_{k \to \infty} \frac{1}{k}  d_p
			\Big(
				{\textrm{Hilb}}_k(h^L_{\mu}, \chi), {\textrm{Hilb}}_k(h^L, \mu)
			\Big)
			\leq
			d_p(h^L_{\mu}, h^L_{K *}).
		\end{equation}
		We see that if $h^L_{\mu} = h^L_{K *}$, then we have $\oplus_{k = 0}^{+ \infty} {\textrm{Hilb}}_k(h^L_{\mu}, \chi) \sim_p \oplus_{k = 0}^{+ \infty} {\textrm{Hilb}}_k(h^L, \mu)$.
		\par 
		Let us now establish that $\mu$ is such that $\oplus_{k = 0}^{+ \infty} {\textrm{Hilb}}_k(h^L, \mu) \sim_p \oplus_{k = 0}^{+ \infty} {\textrm{Hilb}}_k(h^L_{\mu}, \chi)$ for some $p \in [1, +\infty[$, then $\mu$ is weakly Bernstein-Markov with respect to $(K, h^L)$ if and only if $h^L_{\mu} = h^L_{K *}$, which would thereby finish the proof of Proposition \ref{prop_weak_bm}.
		From Proposition \ref{prop_d_p_and_vol} and the trivial bound  ${\textrm{Ban}}_k^{\infty}(K, h^L) \geq {\textrm{Hilb}}_k(h^L, \mu)$, we see that $\mu$ is weakly Bernstein-Markov with respect to $(K, h^L)$ if and only if $\oplus_{k = 0}^{+ \infty} {\textrm{Hilb}}_k(h^L, \mu) \sim_1 \oplus_{k = 0}^{+ \infty} {\textrm{Ban}}_k^{\infty}(K, h^L)$.
		If $\mu$ is such that $\oplus_{k = 0}^{+ \infty} {\textrm{Hilb}}_k(h^L, \mu) \sim_p \oplus_{k = 0}^{+ \infty} {\textrm{Hilb}}_k(h^L_{\mu}, \chi)$ for some $p \in [1, +\infty[$, by (\ref{eq_ban_hilb_etc}), we see that $\oplus_{k = 0}^{+ \infty} {\textrm{Hilb}}_k(h^L, \mu) \sim_1 \oplus_{k = 0}^{+ \infty} {\textrm{Ban}}_k^{\infty}(K, h^L)$ if and only if $\oplus_{k = 0}^{+ \infty} {\textrm{Hilb}}_k(h^L_{\mu}, \chi) \sim_1 \oplus_{k = 0}^{+ \infty} {\textrm{Hilb}}_k(h^L_{K *}, \chi)$.
		The latter condition is equivalent to  $h^L_{\mu} = h^L_{K *}$ by (\ref{eq_berndt}) and the fact that $d_1$ separates points.
	\end{proof}
	\end{sloppypar}
	
	Due to the reasons explained after (\ref{eq_bound_below_measure}), we introduce the following definition.
	
	\begin{defn}\label{defn_approx}
		For a bounded metric $h^L$ on $L$ with a psh potential, a non-pluripolar measure $\mu$ is said to be \textit{approximable} if for any (or some) sequence of positive smooth metrics $h^L_i$ on $L$ which increase towards $h^L$, as $i \to \infty$, and for any $p \in [1, +\infty[$, the following holds
		\begin{equation}\label{eq_approx_defn}
			\lim_{i \to \infty}
			\limsup_{k \to \infty}
			\frac{1}{k}
			d_p
			\Big(
				{\textrm{Hilb}}_k(h^L, \mu), {\textrm{Hilb}}_k(h^L_i, \mu)
			\Big)
			=
			0.
		\end{equation}	
		We say that $\mu$ is \textit{approximable} if it is approximable for any $h^L$ as above.
	\end{defn}
	\begin{rem}
		By Theorem \ref{thm_asympt_bm} and Proposition \ref{prop_d_p_and_vol}, it is enough to verify (\ref{eq_approx_defn}) for $p = 1$.
		Also, immediately from Dini's theorem, any non-pluripolar measure $\mu$ is approximable for an arbitrary continuous metric $h^L$ on $L$.
	\end{rem}
	By (\ref{eq_berndt}) and (\ref{eq_dist_conv}), we see that the Lebesgue measure on $X$ is approximable.
	By repeating the proof of Theorem \ref{thm_equiv_2}, the reader will verify the following result.
	\begin{prop}\label{prop_approx_mult_linf}
		If $\mu$ is approximable, then for any $f \in L^{\infty}(K, \mu)$, $f \geq 0$, where $K$ is the support of $\mu$, the measure $f \cdot \mu$ is also approximable.
	\end{prop}
	\begin{prop}
		For a continuous metic $h^L$ on $L$, the following statements are equivalent: 
		\begin{enumerate}[a)]
			\item The measure $\mu$ is approximable for $h^L_{\mu}$.
			\item We have $\oplus_{k = 0}^{+ \infty} {\textrm{Hilb}}_k(h^L, \mu) \sim_p \oplus_{k = 0}^{+ \infty} {\textrm{Hilb}}_k(h^L_{\mu}, \chi)$, for a volume form $\chi$ on $X$.
		\end{enumerate}
	\end{prop}
	\begin{proof}
		The implication $a) \Rightarrow b)$ follows from (\ref{eq_bound_below_measure}) by the same argument as after (\ref{eq_dp_hilb_lebenv2}).
		\par 
		Let us now establish the implication $b) \Rightarrow a)$.
		Let $h^L_i$ be a sequence of metrics on $L$ as in Definition \ref{defn_approx}.
		Remark first that the sequence of non-negligible psh envelopes $h^L_{i, \mu}$, associated with $\mu$ and $h^L_i$, increases towards $h^L_{\mu}$.
		To see this, it is immediate that $h^L_{i, \mu}$ increase in $i \in \nat$, and we always have $h^L_{i, \mu} \leq h^L_{\mu}$.
		Hence, by (\ref{eq_phi_mu_form}), it is enough to establish that $\lim_{i \to \infty} h^L_{i, \mu} \geq h^L_{K \setminus E}$ for a certain $E \subset K$, verifying $\mu(E) = 0$.
		By Proposition \ref{prop_env_gen_fact}, we consider $E_i \subset K$, $\mu(E_i) = 0$, so that $h^L_{i, \mu}$ coincides with the psh envelope $h^L_{i, K \setminus E_i}$ associated with $h^L_i$ and $K \setminus E_i$.
		We let $E = \cup E_i$. We trivially have $\mu(E) = 0$, and by Remark \ref{rem_leb_env_id}, we also deduce $h^L_{i, \mu} = h^L_{i, K \setminus E}$.
		However, it is standard, cf. \cite[Proposition 2.2.2]{GuedjLuZeriahEnv}, that $h^L_{i, K \setminus E}$ increase towards $h^L_{K \setminus E}$, which finishes the proof.
		\par 
		It is then possible to chose a sequence of smooth positive metrics $h^L_{0, i}$ on $L$ such that for any $i \in \nat$, we have $h^L_{i, \mu} \geq h^L_{0, i}$, and $h^L_{0, i}$ increase towards $h^L_{\mu}$.
		By using Theorem \ref{thm_asympt_bm}, $b)$ and the Lidskii inequality as in (\ref{eq_dp_hilb_lebenv}), we establish that for an arbitrary volume form $\chi$ on $X$, we have
		\begin{equation}\label{eq_approx_defn2}
			\limsup_{k \to \infty}
			\frac{1}{k}
			d_p
			\Big(
				{\textrm{Hilb}}_k(h^L, \mu), {\textrm{Hilb}}_k(h^L_i, \mu)
			\Big)
			\leq
			\limsup_{k \to \infty}
			\frac{1}{k}
			d_p
			\Big(
				{\textrm{Hilb}}_k(h^L_\mu, \chi), {\textrm{Hilb}}_k(h^L_{i, 0}, \chi)
			\Big).
		\end{equation}	
		We deduce $a)$ from (\ref{eq_berndt}), (\ref{eq_dist_conv}), (\ref{eq_approx_defn2}) and the fact that $h^L_{i, 0}$ increases towards $h^L_{\mu}$.
	\end{proof}
	\par 
	Recall that a regular envelope, $Q(\cdot)$, was defined in (\ref{eq_reg_env}).
	\begin{prop}
		A non-pluripolar Borel probability measure $\mu$ on $X$ is approximable for any bounded metric $h^L$, verifying $Q(h^L_{\mu})_* = h^L_{\mu}$.
		In particular, if $h^L$ is $\mu$-pluriregular, then $\mu$ is approximable for $h^L$.
	\end{prop}
	\begin{proof}
		From (\ref{eq_bound_below_measure}), we immediately have 
		\begin{equation}\label{eq_linf_bnd_triv}
			{\textrm{Ban}}_k^{\infty}(h^L_{\mu})
			\geq 
			{\textrm{Hilb}}_k(h^L, \mu)
		\end{equation}
		The condition $Q(h^L_{\mu})_* = h^L_{\mu}$ means precisely that $h^L_{\mu}$ is regularizable from above in the notations of \cite{BedfordRegulBelow}, \cite{BouckErik21}, \cite{FinNarSim}.
		Hence, by \cite[Proposition 2.18]{FinNarSim}, for an arbitrary volume form $\chi$ on $X$, the norm $\oplus_{k = 0}^{+\infty} {\textrm{Ban}}_k^{\infty}(h^L_{\mu})$ is $p$-equivalent to $\oplus_{k = 0}^{+\infty} {\textrm{Hilb}}_k(h^L_{\mu}, \chi)$ for any $p \in [1, +\infty[$.
		From this, (\ref{eq_linf_bnd_triv}) and the Lidskii inequality, we deduce that the analogue of (\ref{eq_approx_defn2}) holds.
		The proof then follows by the same argument as after (\ref{eq_approx_defn2}).
	\end{proof}

	\section{Mabuchi geometry and the contact set}\label{sect_mabuch}
	The main goal of this section is to recall the basics of Mabuchi geometry and to establish Propositions \ref{prop_leb_negl} and \ref{prop_local_geod}.
	Throughout the section, we denote by $\mathbb{D}(a, b)$ the complex annulus with inner radius $a$ and outer radius $b$, and by $\pi : X \times \mathbb{D}(e^{-1}, 1) \to X$ and $z : X \times \mathbb{D}(e^{-1}, 1) \to \mathbb{D}(e^{-1}, 1)$ the usual projections.
	\par 
	Mabuchi in \cite{Mabuchi} introduced a certain metric on the space of positive metrics on $L$, the geodesics of which admit the description as solutions to a certain homogeneous Monge-Ampère equation.
	To recall this, upon fixing a reference positive metric $h^L_*$ on $L$, we identify the space of positive metrics on $L$ with the space of Kähler potentials of $\omega := 2 \pi c_1(L, h^L_*)$ as in (\ref{eq_pot_metr_corr}).
	On the space of Kähler potentials $\xi : X \to \real$ of $\omega$, for $p \in [1, +\infty[$, we introduce the following Finsler metrics
	\begin{equation}\label{eq_finsl_dist_fir}
		\| \xi \|_p^u
		:=
		\sqrt[p]{
		\frac{1}{\int_X \omega^n}
		 \int_X |\xi(x)|^p \cdot \omega_u^n(x)}.
	\end{equation}
	The path length metric structure associated with (\ref{eq_finsl_dist_fir}) on the space of Kähler potentials was introduced by Mabuchi \cite{Mabuchi} for $p = 2$, and by Darvas \cite{DarvasFinEnerg} for any $p \in [1, +\infty[$.
	\par 
	To describe the geodesics in this space, we identify paths $u_t$, $t \in [0, 1]$, of Kähler potentials with rotationally-invariant $\hat{u} : X \times \mathbb{D}(e^{-1}, 1) \to \real$, as follows
	\begin{equation}\label{eq_defn_hat_u}
		\hat{u}(x, \tau) = u_{t}(x), \quad \text{where} \quad x \in X \, \text{ and } \, t = - \log |\tau|.
	\end{equation}
	According to \cite{Semmes}, \cite{DonaldSymSp} smooth geodesic segments in Mabuchi space can be described as the only path of Kähler potentials $u_t$, $t \in [0, 1]$, connecting $u_0$ to $u_1$, so that $\hat{u}$ is the solution of the Dirichlet problem associated with the homogeneous Monge-Ampère equation
	\begin{equation}\label{eq_ma_geod}
		(\pi^* \omega + \imun \partial \dbar \hat{u})^{n + 1} = 0,
	\end{equation}
	with boundary conditions $\hat{u}(x, e^{\imun \theta}) = u_0(x)$, $\hat{u}(x, e^{-1 + \imun \theta}) = u_1(x)$, $x \in X, \theta \in [0, 2\pi]$.
	By the work of X. Chen \cite{ChenGeodMab} and later compliments by B{\l}ocki \cite{BlockiGeod} and Chu-Tosatti-Weinkove \cite{ChuTossVeinC11}, we now know that $\mathscr{C}^{1, 1}$ solutions to (\ref{eq_ma_geod}) always exist.
	\par 
	Berndtsson in \cite[\S 2.2]{BernBrunnMink} proved that for $u_0, u_1 \in {\rm{PSH}}(X, \omega) \cap L^{\infty}(X)$,  \textit{weak solutions} to (\ref{eq_ma_geod}) exist, i.e. (\ref{eq_ma_geod}) has solutions when the wedge power is interpreted in Bedford-Taylor sense \cite{BedfordTaylor} and the boundary conditions mean that $\sup_X | u_{\epsilon} - u_0 | \to 0$ and $\sup_X | u_{1 - \epsilon} - u_1 | \to 0$, as $\epsilon \to 0$.
	From \cite[(2.1)]{BernBrunnMink}, the solution $\hat{u}(u_0, u_1)$ can be described through the following envelope construction:
	\begin{multline}\label{eq_env_geod}
		\hat{u}(u_0, u_1) := \sup \Big\{
		 	\hat{u} \in {\rm{PSH}}(X \times \mathbb{D}(e^{-1}, 1), \pi^* \omega) : \hat{u} \text{ is $S^1$-invariant}
		 	\\ \text{and } \lim_{t \to 0} \hat{u}(x, e^{-t}) \leq u_0, \lim_{t \to 1} \hat{u}(x, e^{-t}) \leq u_1 
		 \Big\}.
	\end{multline}
	\par 
	For the corresponding path $u_t$, $t \in [0, 1]$, we then have $u_t \in {\rm{PSH}}(X, \omega) \cap L^{\infty}(X)$.
	Since $u_t(x)$ is convex in $t \in [0, 1]$ for any $x \in X$, cf. \cite[Theorem I.5.13]{DemCompl}, and $u_t(x)$ is continuous at $t = 0, 1$ by \cite[p. 7]{BernBrunnMink} (and automatically at $t \in ]0, 1[$ by convexity) the derivative at $t = 0$, which we denote by $\dot{u}_0$, is well-defined, and from \cite[\S 2.2]{BernBrunnMink}, we know that $\dot{u}_0$ is bounded.
	\begin{prop}\label{prop_mab_contact_set}
		Assume that for $u_i \in {\rm{PSH}}(X, \omega) \cap L^{\infty}(X)$, $i = 1, 2$, we have $u_0 \leq u_1$.
		Then $\dot{u}_0 \geq 0$.
		Also, if for a given $x \in X$, we have $u_0(x) = u_1(x)$, then $\dot{u}_0(x) = 0$.
		Inversely, if $u_0$ is strictly $\omega$-psh, and for a given $x \in X$, there is $\epsilon > 0$ and an open neighborhood $U$ of $x$, so that for any $y \in U$, we have $u_0(y) \leq u_1(y) - \epsilon$, then we have $\dot{u}_0(x) > 0$.
	\end{prop}
	\begin{proof}
		Remark first that by our condition $u_0 \leq u_1$, $\hat{u}(x, \tau) := u_0(x)$, $\tau \in \mathbb{D}(e^{-1}, 1)$, $x \in X$, is one of the contenders in (\ref{eq_env_geod}).
		Hence, we deduce that the Mabuchi geodesic $u_t$, $t \in [0, 1]$, between $u_0$ and $u_1$ satisfies $u_t \geq u_0$ for any $t \in [0, 1]$.
		Directly from this, we get $\dot{u}_0 \geq 0$.
		\par 
		Now, by the convexity of $u_t$ in $t$, we deduce that 
		\begin{equation}\label{eq_conv_geod}
			u_t \leq (1 - t)u_0 + t u_1,
		\end{equation}
		which implies that $\dot{u}_0 \leq u_1 - u_0$.
		Directly from this, we get the second statement.
		\par 
		To establish the third statement, we proceed as follows.
		Let $\rho : X \to [0, 1]$ be an arbitrary non-negative function with support inside of $U$, such that $\rho(x) = 1$ for the point $x$ as in the statement of Proposition \ref{prop_mab_contact_set}.
		Let $g : [0, 1] \to \real$, be an arbitrary function such that $\tau \mapsto g(- \log |\tau|)$, $\tau \in \mathbb{D}(e^{-1}, 1)$, is strictly psh on $\mathbb{D}(e^{-1}, 1)$ and verifies $g(0) = 0$ (e.g. $g(t) = \exp(-2 t) - 1$).
		We now verify that there is $\epsilon_0 > 0$, such that for any $0 < \epsilon < \epsilon_0$, the function
		\begin{equation}\label{eq_bump_const}
			\hat{u}(x, \tau) 
			:=
			u_0(x)
			-
			\epsilon^2 \cdot \rho(x) \cdot \log |\tau| + \epsilon^3 g(- \log |\tau|) + \epsilon^3 \cdot \log |\tau| \cdot g(1),
		\end{equation}
		is one of the contenders in (\ref{eq_env_geod}).
		We first verify that $\hat{u} \in {\rm{PSH}}(X \times \mathbb{D}(e^{-1}, 1), \pi^* \omega)$. 
		Since the first and the two last terms in (\ref{eq_bump_const}) are $\pi^* \omega$-psh, it is enough to verify that $\hat{u}$ is psh in $U \times \mathbb{D}(e^{-1}, 1)$, where the support of the second term is localized.
		By writing in a local chart and using the fact that $u_0$ is strictly psh, we see that the fact that $\hat{u}$ is psh for $\epsilon > 0$ small enough amounts essentially to the fact that for any $C \in \real$, there is $\epsilon_0 > 0$, such that for any $0 < \epsilon < \epsilon_0$, the form $\imun dz \wedge d \overline{z} + C \epsilon^2 \imun dz \wedge d \overline{\tau} - C \epsilon^2 \imun d\overline{z} \wedge d \tau + \epsilon^3 \imun d \tau \wedge d \overline{\tau}$ is positive, which is a trivial verification.
		Let us now verify that the boundary conditions from (\ref{eq_env_geod}) are satisfied.
		We have $\lim_{t \to 0} \hat{u}(x, e^{-t}) \leq u_0$ by the assumption $g(0) = 0$.
		Moreover, since the sum of the last two terms in (\ref{eq_bump_const}) is zero for $|\tau| = e^{-1}$, we have $\lim_{t \to 1} \hat{u}(x, e^{-t}) = u_0 + \epsilon^2 \cdot \rho(z)$.
		But since $\rho$ has support in $U$, for $\epsilon > 0$ small enough $u_0 + \epsilon^2 \cdot \rho(z) \leq u_1$ by our standing assumption.
		\par 
		From the above, we get $\hat{u}(u_0, u_1) \geq \hat{u}$.
		By choosing $\epsilon > 0$ small enough and comparing the derivatives at $t = 0$ of the respective paths of metrics, we deduce that $\dot{u}_0(x) > 0$ as required.
	\end{proof}
	
	\par 
	To apply this in the setting related with psh envelopes, we need the following result.
	\begin{prop}\label{prop_contact_set_closed}
		For any point $x \in X$ outside the closure of a non-pluripolar subset $E$ and any metric $h^L$ on $L$ with strictly psh potential, there is neighborhood $U$ of $x$ and $\epsilon > 0$, such that over $U$, we have $h^L \geq \exp(\epsilon) \cdot h^L_E$.
	\end{prop}
	\begin{proof}
		Let $\rho$ be an arbitrary non-negative function with support inside of $U$, such that $\rho(x) = 1$ for the point $x$ as in the statement.
		Since $h^L$ has a strictly psh potential, there is $\epsilon > 0$, small enough so that $h^L \cdot \exp(- \epsilon \rho)$ has a psh potential.
		Directly from the definition of $h^L_E$, we then deduce $h^L(x) \geq \exp(\epsilon) \cdot h^L_E(x)$.
	\end{proof}
	\begin{sloppypar}
	\begin{proof}[Proof of Proposition \ref{prop_leb_negl}.]
		Directly from Propositions \ref{prop_mab_contact_set}, \ref{prop_contact_set_closed} and Remark \ref{rem_leb_env_id}, we deduce that $\phi(h^L, NZ)$ is non-positive, $\phi(h^L, NZ) = 0$ over the set $d(NZ)$ of density points of $NZ$ and $\phi(h^L, NZ) < 0$ over $X \setminus K$.
		By this, the Lebesgue's density theorem and the assumption $\lambda(Z) = 0$, we deduce Proposition \ref{prop_leb_negl}.
	\end{proof}
	\end{sloppypar}
	\begin{proof}[Proof of Proposition \ref{prop_local_geod}]
		From Propositions \ref{prop_contact_set_closed} and (\ref{eq_conv_geod}), we conclude that the Mabuchi geodesic  $h^L_t$, (resp. $h^{L'}_t$), $t \in [0, 1]$, between $h^L_0 := h^L$ and $h^L_1 := h^L_{{\rm{Leb}}, NZ}$ (resp. and $h^{L'}_0 := h^{L'}$ and $h^{L'}_1 := h^{L'}_{{\rm{Leb}}, K'}$) verify $h^L_t = h^L$ away from $U$, (resp. $h^{L'}_t = h^L$ away from $U'$).
		By (\ref{eq_env_geod}), we see then directly that $p^* h^{L'}_t = h^L_t$, which implies immediately Proposition \ref{prop_local_geod}. 
	\end{proof}
	
	\section{Growth of balls of holomorphic sections and Lebesgue envelopes}\label{sect_growth_balls}
	The main goal of this section is to apply Theorem \ref{thm_equiv_2} to study the growth of balls of holomorphic sections associated with $L^2$-norms supported on measurable subsets.
	\par 
	For this, we recall that the \textit{Monge-Ampère energy functional} $\mathscr{E}$, defined on the space of Kähler potentials $u : X \to \real$, associated with a Kähler form $\omega$ on $X$, is the unique functional (up to a constant) that for any Kähler potentials $u, v$ satisfies the following equation:
	\begin{equation}
		\frac{d}{dt}
		\mathscr{E}((1 - t) u + t v)|_{t = 0}
		=
		\frac{
		\int_X (v - u) \cdot \omega_u^n}{\int_X \omega^n}.
	\end{equation}
	While $\mathscr{E}$ is only well-defined up to a constant, the differences $\mathscr{E}(u) - \mathscr{E}(v)$ are well-defined and can be explicitly evaluated as follows
	\begin{equation}\label{eq_energy}
		\mathscr{E}(u) - \mathscr{E}(v)
		=
		\frac{1}{(n + 1) \int_X \omega^n}
		\sum_{j = 0}^{n} \int_X (u - v) w_u^j \wedge w_v^{n - j}.
	\end{equation}
	By \cite[Proposition 10.14]{GuedjZeriahBook}, $\mathscr{E}$ is monotonic, i.e. for any $u \leq v$, we have $\mathscr{E}(u) \leq \mathscr{E}(v)$.
	From this and Demailly regulartization theorem \cite{DemRegul}, it is reasonable to extend the domain of the definition of $\mathscr{E}$ to ${\rm{PSH}}(X, \omega)$ as
	\begin{equation}\label{eq_energy_ext}
		\mathscr{E}(u) :=
		\inf 
		\Big \{ 
		 \mathscr{E}(v) : v \text{ is a $\omega$-Kähler potential, verifying } u \leq v
		\Big \}.
	\end{equation}
	Remark, that $\mathscr{E}$ can take the value $- \infty$ on non-smooth elements of ${\rm{PSH}}(X, \omega)$, but by monotonicity, it takes finite values on ${\rm{PSH}}(X, \omega) \cap L^{\infty}(X)$, see also \cite{DarvasFinEnerg} for a more precise result.
	Below, we extend the definition of $\mathscr{E}$ to metrics on $L$ with bounded psh potentials by applying the energy functional to the potential of the metric.
	\par 
	\begin{thm}\label{thm_balls}
		For any continuous metric $h^L$ on $L$ with a psh potential, and an arbitrary measurable subset $A \subset X$, which is Lebesgue non-negligible, the following identity holds
		\begin{equation}\label{eq_balls}
			\lim_{k \to \infty} \frac{\log {\rm{vol}} ({\textrm{Hilb}}_k(A, h^L)) - \log {\rm{vol}} ({\textrm{Hilb}}_k(h^L))}{k \cdot \dim H^0(X, L^{\otimes k})} 
			=
			\mathscr{E}(h^L_{{\rm{Leb}}, A})
			-
			\mathscr{E}(h^L).
		\end{equation}
	\end{thm}
	\begin{rem}
		When the subset $A$ is a closed subset $K \subset X$, such that $(K, h^L)$ is pluriregular and the Lebesgue measure on $K$ is determining for $(K, h^L)$, the result was previously established by Berman-Boucksom \cite{BermanBouckBalls}, with $h^L_K$  in place of $h^L_{{\rm{Leb}}, A}$ on the right-hand side of (\ref{eq_balls}).
		Our results are compatible by Proposition \ref{prop_plurireg}.
	\end{rem}
	\begin{proof}
		Remark that we trivially have ${\textrm{Hilb}}_k(A, h^L) \leq {\textrm{Hilb}}_k(h^L)$.
		From this and (\ref{eq_vol_d1}), we deduce
		\begin{equation}\label{eq_d_1_vol}
			\frac{\log {\rm{vol}} ({\textrm{Hilb}}_k(A, h^L)) - \log {\rm{vol}} ({\textrm{Hilb}}_k(h^L))}{\dim H^0(X, L^{\otimes k})} 
			=
			d_1
			\Big(
			{\textrm{Hilb}}_k(A, h^L), {\textrm{Hilb}}_k(h^L)
			\Big).
		\end{equation}
		By Theorem \ref{thm_asympt_bm} and the fact that $d_1$ satisfies the triangle inequality, cf. \cite[Theorem 1.1]{DarvLuRub}, we have
		\begin{equation}\label{eq_d_1_vol2}
			\limsup_{k \to \infty}
			\frac{1}{k}
			d_1
			\Big(
			{\textrm{Hilb}}_k(A, h^L), {\textrm{Hilb}}_k(h^L)
			\Big)
			=
			\limsup_{k \to \infty}
			\frac{1}{k}
			d_1
			\Big(
			{\textrm{Hilb}}_k(h^L_{{\rm{Leb}}, A}, \chi), {\textrm{Hilb}}_k(h^L)
			\Big),
		\end{equation}
		where $\chi$ is an arbitrary smooth volume form. 
		Also, by (\ref{eq_berndt}), we have
		\begin{equation}\label{eq_d_1_vol3}
			\limsup_{k \to \infty}
			\frac{1}{k}
			d_1
			\Big(
			{\textrm{Hilb}}_k(h^L_{{\rm{Leb}}, A}, \chi), {\textrm{Hilb}}_k(h^L)
			\Big)
			=
			d_1(h^L_{{\rm{Leb}}, A}, h^L).
		\end{equation}
		\par 
		Darvas proved in \cite[Corollary 4.14]{DarvasFinEnerg} that for any bounded metrics $h^L_0$, $h^L_1$ with psh potentials, verifying $h^L_0 \leq h^L_1$, we have
		\begin{equation}\label{eq_d1_ener}
			d_1(h^L_0, h^L_1) = \mathscr{E}(h^L_0) - \mathscr{E}(h^L_1).
		\end{equation}
		We finish the proof by (\ref{eq_d_1_vol})-(\ref{eq_d1_ener}) and the trivial bound $h^L_{{\rm{Leb}}, A} \leq h^L$.
	\end{proof}
	\par 

	\section{Generalized Toeplitz operators and Toeplitz matrices}\label{sect_gen_toepl}
	 The main goal of this section is to generalize the main results of this article to the setting of generalized Toeplitz operators associated with an arbitrary non-pluripolar Borel measure. 
	 We then make a connection between this generalized setting and the classical theory of Toeplitz matrices.
	 \par 
	 We consider a non-pluripolar probability Borel measure $\mu$ on $X$ with support $K$.
	 Let $f \in L^{\infty}(K, \mu)$ be a fixed function, with $K'$ denoting its essential support.
	 For brevity, we let $\mu' := f \cdot \mu$.
	 Remark that $\mu'$ is non-pluripolar.
	 \par 
	 For $k \in \nat^*$, we denote by $T_k(f, \mu) \in {\enmr{H^0(X, L^{\otimes k})}}$ the \textit{Generalized Toeplitz operator with symbol} $f$, i.e. $T_k(f, \mu) := B_k(\mu) \circ M_k(f)$, where $B_k(\mu) : L^{\infty}(K, L^{\otimes k}) \to H^0(X, L^{\otimes k})$ is the orthogonal projection to $H^0(X, L^{\otimes k})$ with respect to ${\textrm{Hilb}}_k(h^L, \mu)$, and $M_k(f) : H^0(X, L^{\otimes k}) \to L^{\infty}(K, L^{\otimes k})$ is the restriction on $K$, composed with a multiplication map by $f$.
	 \par 
	 \begin{thm}\label{thm_toepl_gen}
	 	For any $f \in L^{\infty}(K, \mu)$, $f \neq 0$, there is $c > 0$ such that for any $\lambda_{\min}(T_k(f, \mu)) \geq \exp(- ck)$ for $k \in \nat$ big enough.
	 	\par 
	 	We will now assume that $(K, h^L)$ (resp. $(K', h^L)$) is $\mu$-pluriregular (resp. $\mu'$-pluriregular) and $\mu(K' \cap f^{-1}(0)) = 0$.
		We let $c(f) := \max_{x \in X} \log(h^L_{\mu}(x) / h^L_{\mu'}(x))$.
		Then for any $\epsilon > 0$, there is $k_0 \in \nat$, such that for any $k \geq k_0$, we have
		\begin{equation}\label{eqthm_toepl_gen1}
			\exp(- (c(f) + \epsilon) k) \leq \lambda_{\min}(T_k(f, \mu)) \leq \exp(- (c(f) - \epsilon) k).
		\end{equation}
		Moreover, if we denote by $\delta[\cdot]$ the Dirac mass, then the sequence of probability measures
		\begin{equation}\label{eqthm_toepl_gen2}
			\eta_k := 
			\lim_{k \to \infty} \frac{1}{\dim H^0(X, L^{\otimes k})} \sum_{\lambda \in {\rm{Spec}} (T_k(f, \mu))}  \delta \Big[ -\frac{\log(\lambda)}{k} \Big]
		\end{equation}
		converges weakly to the (unique) probability measure $\eta$ on $\real$, verifying $\int_{\real} x^p d \eta(x) = d_p(h^L_{\mu}, h^L_{\mu'})$, where $d_p$ is the Darvas distance from (\ref{sect_mabuch}).
	\end{thm}
	\begin{sloppypar}
	\begin{proof}
		The first statement is equivalent to the fact that there is $c > 0$, such that for any $k \in \nat^*$,
		\begin{equation}
			{\textrm{Hilb}}_k(h^L, \mu') \geq \exp(-ck) \cdot {\textrm{Hilb}}_k(h^L, \mu),
		\end{equation}
		which follows immediately from Theorem \ref{thm_asympt_bm}, as it shows that we have ${\textrm{Hilb}}_k(h^L, \mu') \geq \exp(-c k) \cdot {\textrm{Ban}}_k^{\infty}(h^L)$ for a certain $c > 0$ and $k \in \nat^*$.
		\par 
		We note that by Proposition \ref{prop_plurireg} and Theorem \ref{thm_bbwn_determining}, similarly to the proof of Corollary \ref{thm_equiv_1}, for an arbitrary volume form $\chi$ on $X$, we have
		\begin{equation}
		\begin{aligned}
			&
			\oplus_{k = 0}^{\infty} {\textrm{Hilb}}_k(h^L, \mu) \sim \oplus_{k = 0}^{\infty} {\textrm{Hilb}}_k(h^L_{\mu}, \chi),
			\\
			&
			\oplus_{k = 0}^{\infty} {\textrm{Hilb}}_k(h^L, \mu') \sim \oplus_{k = 0}^{\infty} {\textrm{Hilb}}_k(h^L_{\mu'}, \chi).
		\end{aligned}
		\end{equation}
		From this point, the proof of all the other statements proceed in an identical manner with the proofs of Theorems \ref{thm_regul}, \ref{thm_distr} and Corollary \ref{cor_exist}.
	\end{proof}
	\end{sloppypar}
	\begin{rem}
		Immediately from the proof and Proposition \ref{prop_approx_mult_linf}, we see that for (\ref{eqthm_toepl_gen2}), instead of pluriregularity of $\mu$ and $\mu'$, it suffices to require that the measure $\mu$ is approximable.
	\end{rem}
	\par 
	We will now specialize the above theorem in the setting of Toeplitz matrices.
	More specifically, let $\mathbb{S}^1$ be the unit circle and $f \in L^{\infty}(\mathbb{S}^1)$, $f \neq 0$, be a fixed real function, which we can be written as $f(\theta) = a_0 + \sum_{i = - \infty}^{+ \infty} a_j \exp(\imun j \theta)$, where $\theta \in [0, 2 \pi[$, gives a standard parametrization of $\mathbb{S}^1$. 
	Then $a_i = \overline{a}_{-i}$ and not all $a_i$ vanish.
	\par 
	Now, we embed $\mathbb{S}^1$ in $X := \mathbb{P}^1$ as one of the great circles (for concreteness given by $\theta \mapsto [1: \exp(i \theta)] \in \mathbb{P}^1$, $\theta \in [0, 2 \pi[$, where $[1 : z] \in \mathbb{P}^1$, $z \in \comp$, is a standard affine chart) and denote by $\mu$ the Lebesgue measure on $\mathbb{S}^1$, viewed as a measure on $X$.
	Remark that as $\mathbb{S}^1$ is totally real, the measure $\mu$ is non-pluripolar, see \cite{SadullaevReal}, cf. \cite[Exercise 4.39.5]{GuedjZeriahBook}. 
	We then take $L := \mathscr{O}(1)$, and in a standard basis of monomials of $H^0(X, L^{\otimes k})$ (given by $z^i$, $i = 0, \ldots, k$, in the already mentioned affine chart), the operator $T_k(f, \mu)$ writes as the following Toeplitz matrix 
	\begin{equation}
		T_k[a]
		:=
		\begin{bmatrix}
		a_0 & a_{-1} & a_{-2} & \cdots & a_{-k} \\
		a_1 & a_0 & a_{-1} & \cdots & a_{-k+1} \\
		a_2 & a_1 & a_0 & \cdots & a_{-k+2} \\
		\vdots & \vdots & \vdots & \ddots & \vdots \\
		a_k & a_{k-1} & a_{k-2} & \cdots & a_0
		\end{bmatrix}.
	\end{equation}
	Remark that the matrix is Hermitian by our assumption on $f$.
	\par 
	The study of Toeplitz matrix spectra was profoundly influenced by the seminal works of Szegő \cite{SzegoFST}, \cite{SzegoSND}. 
	For an overview of this extensive field, see \cite{ToeplitzSurvey}, \cite{NikolskiBook}. The smallest eigenvalue has also been a central topic in this theory, as explored in \cite{WidomToepl}, \cite{KacMurdockSzego}, \cite{ParterToepl}.
	However, to the best of the author's knowledge, its exponential decay has not yet been thoroughly investigated.
	The two results below shed some light on this question.
	\par 
	The following result is an immediate consequence of the above interpretation of Toeplitz matrices, Theorem \ref{thm_toepl_gen} and the fact that $(\mathbb{S}^1, h^L)$ is $\mu$-pluriregular in $\mathbb{P}^1$, cf. \cite[Theorem 2.22]{BermanPriors}.
	\begin{cor}
		The smallest eigenvalue $\lambda_{\min}(T_k[a])$ of $T_k[a]$ decays at most exponentially, as $k \to \infty$, i.e. there is $c > 0$ such that $\lambda_{\min}(T_k[a]) \geq \exp(- ck)$ for any $k \in \nat$ big enough.
	\end{cor}	
	We will now assume that $K = \esssupp f$ consists of a finite union of intervals and the Lebesgue measure of the set $K \cap f^{-1}(0)$ is zero.
	We denote by $h^L_{\mathbb{S}^1}$ and $h^L_K$ the psh envelopes associated with the Fubini-Study metric on $L$ and the sets $\mathbb{S}^1$ and $K$ respectively, both viewed as subsets in $\mathbb{P}^1$.
	The following result is an immediate consequence of Theorem \ref{thm_toepl_gen} and the fact that $(K, h^L)$ is $\mu'$-pluriregular in $\mathbb{P}^1$, which follows from \cite[Theorem 2.22]{BermanPriors}.
	\begin{cor}
		We let $c(f) := \max_{x \in \mathbb{P}^1} \log(h^L_{\mathbb{S}^1}(x) / h^L_K(x))$.
		Then for any $\epsilon > 0$, there is $k_0 \in \nat$, such that for any $k \geq k_0$, we have
		\begin{equation}\label{eqthm_toepl_gen12}
			\exp(- (c(f) + \epsilon) k) \leq \lambda_{\min}(T_k[a]) \leq \exp(- (c(f) - \epsilon) k).
		\end{equation}
		Moreover, the sequence of probability measures
		\begin{equation}\label{eqthm_toepl_gen22}
			\eta_k := 
			\lim_{k \to \infty} \frac{1}{k + 1} \sum_{\lambda \in {\rm{Spec}} (T_k[a])}  \delta \Big[ -\frac{\log(\lambda)}{k} \Big]
		\end{equation}
		converges weakly to the (unique) probability measure $\eta$ on $\real$, verifying $\int_{\real} x^p d \eta(x) = d_p(h^L_{\mathbb{S}^1}, h^L_K)$, for any $p \in [1, +\infty[$, where $d_p$ was defined in Section \ref{sect_mabuch}.
	\end{cor}
	 
\bibliography{bibliography}

		\bibliographystyle{abbrv}

\Addresses

\end{document}

%% file: packages.tex
\usepackage{times}

\usepackage[shortlabels]{enumitem}
\usepackage{enumitem}
\usepackage[utf8]{inputenc}
\usepackage{commath}
\usepackage{amsthm, amscd}
\usepackage{amsmath}
\usepackage{amssymb}
\usepackage{cite}
\usepackage{setspace}
\usepackage{ stmaryrd }
\usepackage{tikz-cd}

\usepackage{tocloft}
\usepackage{sectsty}
\usepackage{xparse}

\newcommand*\tasklabelformat[1]{#1)}

\usepackage{titlesec}

\usepackage{tasks}[newest]
\settasks{
  label = \alph* ,
  label-format = \tasklabelformat ,
  label-width  = 12pt
}
\usepackage{bigints}
\usepackage{comment}
\usepackage{mathtools}
\usepackage{mathrsfs}
\usepackage{fancyhdr}

\allowdisplaybreaks[4]

\numberwithin{equation}{section}
\usepackage{etoolbox}
\patchcmd{\thebibliography}
  {\settowidth}
  {\setlength{\itemsep}{0pt plus -10pt}\settowidth}
  {}{}
\apptocmd{\thebibliography}
  {
  }
  {}{}

%% file: theorem.tex
\makeatletter
\newtheorem*{rep@theorem}{\rep@title}
\newcommand{\newreptheorem}[2]{%
\newenvironment{rep#1}[1]{%
 \def\rep@title{#2 \ref{##1}}%
 \begin{rep@theorem}}%
 {\end{rep@theorem}}}
\makeatother

\theoremstyle{theorem}

\newreptheorem{theorem}{Theorem}
\newtheorem{thm}{Theorem}[section]
\newtheorem*{thm*}{Theorem}
\theoremstyle{definition}
\newtheorem{prop}[thm]{Proposition}
\newtheorem*{prop*}{Proposition}
\newtheorem{defn}[thm]{Definition}

\newtheorem{cor}[thm]{Corollary}
\newtheorem*{cor*}{Corollary}
\theoremstyle{remark}

\newtheorem{rem}[thm]{Remark}

%% file: title.tex
\title{\vspace*{-1.5cm} Small eigenvalues of Toeplitz operators,
\\
Lebesgue envelopes and Mabuchi geometry}

\author
{Siarhei Finski
}

\date{}

%% file: paddings.tex
\usepackage[%
    left=1in,%
    right=1in,%
    top=1.1in,%
    bottom=0.8in,%
    paperheight=11in,%
    paperwidth=8.5in%
]{geometry}

%% file: commands.tex
\newcommand{\imun} {\sqrt{-1}}

\newcommand{\comp}{\mathbb{C}}
\newcommand{\real}{\mathbb{R}}

\newcommand{\nat}{\mathbb{N}}

\DeclareMathOperator*{\esssup}{ess\,sup}
\DeclareMathOperator*{\essinf}{ess\,inf}
\DeclareMathOperator*{\esssupp}{ess\,supp}

\newcommand{\enmr}[1]{\text{End}{(#1)}}

\newcommand{\ccal}{\mathscr{C}}

\newcommand{\dbar}{ \overline{\partial} }

\newcommand{\scal}[2]{\langle #1, #2 \rangle}

\makeatletter
\DeclareFontFamily{OMX}{MnSymbolE}{}
\DeclareSymbolFont{MnLargeSymbols}{OMX}{MnSymbolE}{m}{n}
\SetSymbolFont{MnLargeSymbols}{bold}{OMX}{MnSymbolE}{b}{n}
\DeclareFontShape{OMX}{MnSymbolE}{m}{n}{
    <-6>  MnSymbolE5
   <6-7>  MnSymbolE6
   <7-8>  MnSymbolE7
   <8-9>  MnSymbolE8
   <9-10> MnSymbolE9
  <10-12> MnSymbolE10
  <12->   MnSymbolE12
}{}
\DeclareFontShape{OMX}{MnSymbolE}{b}{n}{
    <-6>  MnSymbolE-Bold5
   <6-7>  MnSymbolE-Bold6
   <7-8>  MnSymbolE-Bold7
   <8-9>  MnSymbolE-Bold8
   <9-10> MnSymbolE-Bold9
  <10-12> MnSymbolE-Bold10
  <12->   MnSymbolE-Bold12
}{}

\let\llangle\@undefined
\let\rrangle\@undefined
\DeclareMathDelimiter{\llangle}{\mathopen}%
                     {MnLargeSymbols}{'164}{MnLargeSymbols}{'164}
\DeclareMathDelimiter{\rrangle}{\mathclose}%
                     {MnLargeSymbols}{'171}{MnLargeSymbols}{'171}
\makeatother